\documentclass[12pt,reqno]{amsart}
\usepackage[a4paper, left=24mm, right=24mm, top=28mm, bottom=34mm]{geometry}
\usepackage[utf8]{inputenc}
\usepackage[x11names]{xcolor}
\usepackage{amsmath}
\usepackage{amscd}
\usepackage{amssymb}
\usepackage{amsthm}
\usepackage{extarrows}

\usepackage[backend=biber, style=alphabetic, maxnames=5, sorting=nyt,backref=true,doi=false,isbn=false]{biblatex}
 \AtEveryBibitem{\clearfield{url}}
 \AtEveryBibitem{\clearlist{language}}

\usepackage[colorlinks,linkcolor=Brown4, citecolor=Brown4]{hyperref}
\usepackage{url}
\usepackage{xurl}

\usepackage{stmaryrd}
\usepackage{dsfont}
\usepackage{mathrsfs}

\usepackage{mathtools}
\usepackage{comment}
\usepackage{bbm}
\usepackage[shortlabels]{enumitem}
\setlist[enumerate]{leftmargin=*}
\usepackage{tikz-cd}
\usepackage{cleveref}
\usepackage[indent,parfill]{parskip}
\DeclarePairedDelimiterX\Set[1]\lbrace\rbrace{\def\given{\;\delimsize\vert\;}#1}

\newif\ifHideFoot
\HideFoottrue  %
\HideFootfalse  %

\newtheorem{thm}{Theorem}[section]
\newtheorem{lem}[thm]{Lemma}
\newtheorem{cor}[thm]{Corollary}
\newtheorem{prop}[thm]{Proposition}
\newtheorem{conj}[thm]{Conjecture}

\theoremstyle{definition}

\newtheorem{setup}[thm]{Set-up}
\newtheorem{rem}[thm]{Remark}
\newtheorem{defn}[thm]{Definition}
\newtheorem{prob}[thm]{Problem}

\newtheorem{defn-prop}[thm]{Definition-Proposition}
\newtheorem{defn-thm}[thm]{Definition-Theorem}

\newtheorem{ex}[thm]{Example}

\newcommand{\G}{{\mathds G}}
\newcommand{\Q}{{\mathds Q}}
\newcommand{\R}{{\mathds R}}
\newcommand{\Z}{{\mathds Z}}
\newcommand{\C}{{\mathds C}}

\DeclareMathOperator{\Dpst}{\mathit{D}_{\mathrm{pst}}}

\newcommand\End{\mathop{\mathrm{End}}\nolimits}

\newcommand\gr{\text{\rm gr}}

\newcommand\Cl{\mathop{\mathrm{Cl}}\nolimits}

\newcommand\Gal{\mathop{\mathrm{Gal}}\nolimits}

\renewcommand\Im{\mathop{\mathrm{Im}}\nolimits}

\newcommand\id{\mathop{\mathrm{id}}\nolimits}

\newcommand\GL{\mathop{\mathrm{GL}}\nolimits}
\newcommand\Gm{\mathop{\mathbb{G}_\mathrm{m}}\nolimits}

\newcommand\gl{\mathop{\mathfrak{gl}}\nolimits}
\renewcommand\ge{\mathop{\mathfrak{g}}\nolimits}

\newcommand{\fsl}{\mathop{\mathfrak{sl}}\nolimits}
\newcommand\so{\mathop{\mathfrak{so}}\nolimits}

\newcommand\Spin{\mathop{\mathrm{Spin}}\nolimits}
\newcommand\GSpin{\mathop{\mathrm{GSpin}}\nolimits}
\newcommand\SO{\mathop{\mathrm{SO}}\nolimits}

\newcommand\Sym{\mathop{\mathrm{Sym}}\nolimits}

\newcommand\Spec{\mathop{\rm Spec}}

\newcommand\cris{\text{\rm cris}}

\newcommand\KS{\mathop{\mathrm{KS}}\nolimits}

\newcommand\dR{\mathop{\mathrm{dR}}}

\newcommand{\et}{\mathrm{\acute{e}t}}

\DeclareMathOperator{\ad}{ad}

\newcommand{\hk}{hyper-Kähler~}

\newcommand\Kum{\mathop{\mathrm{Kum}}\nolimits}
\newcommand\OG{\mathop{\mathrm{OG}}\nolimits}

\newcommand{\frh}{\mathfrak{h}}
\newcommand{\frg}{\mathfrak{g}}

\newcommand{\rfrg}{\overline{\mathfrak{g}}}

\newcommand{\fg}{\mathfrak{g}}

\newcommand{\cO}{\mathcal{O}}

\newcommand{\cX}{\mathcal{X}}

\newcommand{\scrL}{\mathscr{L}}

\DeclareMathOperator{\tot}{tot}

\DeclareMathOperator{\pst}{pst}
\DeclareMathOperator{\st}{st}

\DeclareMathOperator{\MT}{MT}
\DeclareMathOperator{\ur}{ur}

\DeclareMathOperator{\tw}{tw}

\numberwithin{equation}{section}

\makeatletter
\@namedef{subjclassname@2020}{\textup{2020} Mathematics Subject Classification}
\makeatother

\title[Arithmetic monodromy of hyper-Kähler varieties over $p$-adic fields]{Arithmetic monodromy of hyper-Kähler varieties over $p$-adic fields}
\author[K.~Ito]{Kazuhiro Ito}
\address{Mathematical Institute, Tohoku University, 6-3, Aoba, Aramaki, Aoba-Ku, Sendai, 980-8578, Japan}
\email{kazuhiro.ito.c3@tohoku.ac.jp}

\author[T.~Ito]{Tetsushi Ito}
\address{Department of Mathematics, Faculty of Science, Kyoto University, Kyoto 606-8502, Japan}
\email{tetsushi@math.kyoto-u.ac.jp}

\author[T.~Koshikawa]{Teruhisa Koshikawa}
\address{Research Institute for Mathematical Sciences, Kyoto University, Kyoto 606-8502, Japan}
\email{teruhisa@kurims.kyoto-u.ac.jp}

\author[T.~Takamatsu]{Teppei Takamatsu}
\address{Department of Mathematics (Hakubi center), Faculty of Science, Kyoto University, Kyoto 606-8502, Japan}
\email{teppeitakamatsu.math@gmail.com}

\author[H.~Zou]{Haitao Zou}
\address{Universität Bielefeld, Universitätsstraße 25, 33615 Bielefeld, Germany}
\email{hzou@math.uni-bielefeld.de}


\begin{document}
\begin{abstract}
In this paper, we study the $p$-adic and $\ell$-adic monodromy operators associated with hyper-Kähler varieties over $p$-adic fields, in connection with Looijenga--Lunts--Verbitsky Lie algebras.
We investigate a conjectural relation between the nilpotency indices of these monodromy operators on higher-degree cohomology groups and on the second cohomology, which may be viewed as an arithmetic analogue of Nagai’s conjecture for degenerations of hyper-Kähler manifolds over a disk. We verify this arithmetic version of Nagai’s conjecture for hyper-Kähler varieties over $p$-adic fields, assuming they belong to one of the four known deformation types.
As part of our approach, we introduce a new method to analyze the $p$-adic cohomology of hyper-Kähler varieties via Sen's theory.
\end{abstract}
\maketitle
\setcounter{tocdepth}{1}
\tableofcontents

\section{Introduction} \label{Section:Introduction}

We fix a prime number $p$ throughout the paper.
Let $K$ be a complete discretely valued field of mixed characteristic $(0, p)$ with perfect residue field $k$. Let $\ell$ be a prime number.
Let $X$ be a smooth proper variety over $K$. Then we obtain Galois representations
\[
\phi_{\ell,i}\colon \Gal (\overline{K}/K) \rightarrow \GL ( H^i_{\et} (X_{\overline{K}}, \Q_{\ell}))
\]
for integers $0 \leq i \leq 2 \dim X$.
(Here $\overline{K}$ is an algebraic closure of $K$ and $X_{\overline{K}}\coloneqq X \otimes_K \overline{K}$.)
If $\ell \neq p$, Grothendieck's $\ell$-adic monodromy theorem in \cite{SGA7-I} implies that there is an open subgroup $I$ of the inertia subgroup $I_K \subset \Gal(\overline{K}/K)$ such that the restriction of $\phi_{\ell,i}$ on $I$ is \emph{unipotent}, i.e., there is an integer $\nu \geq 0$ such that 
$(\phi_{\ell,i}(g) -1 )^{\nu+1} =0$ for all $g \in I$.
We call the minimal integer $\nu$ with this property the ($\ell$-adic) monodromy nilpotency index in degree $i$. Denote it by $\nu_{\ell, i}$. If $\ell =p$, a similar monodromy nilpotency index $\nu_{p, i}$ can be defined via Fontaine's formalism on potentially semistable $p$-adic Galois representations.

\begin{prob}\label{Motivation}
How can one determine the monodromy nilpotency index $\nu_{\ell,i}$ for a prime $\ell$ (including $\ell = p$) in a degree $0 \leq i \leq 2 \dim X$?  
\end{prob}

In this paper, we will focus on \Cref{Motivation} in the case that $X$ is a \emph{hyper-Kähler variety}. In complex geometry, hyper-Kähler varieties are also known as irreducible holomorphic symplectic (IHS) varieties, which are higher-dimensional analogues of K3 surfaces and form one of the three building blocks of compact Kähler manifolds with trivial first Chern class.
As with K3 surfaces, the notion of hyper-Kähler variety extends naturally to general base fields, at least in characteristic zero. 
We will see that the rich geometric structure would lead to an ingredient for the \Cref{Motivation} for hyper-Kähler varieties over $K$.

\subsection{Nagai's conjecture and arithmetic analogues}

In studying a projective degeneration of a complex hyper-Kähler variety, Nagai proposed the following conjecture on the nilpotency indices of local monodromy operators:
\begin{conj}[{\cite[Conjecture 5.1]{Nagai}}]
\label{Nagai Conjecture over C}
Denote $\Delta = \Set{z \in \C \given \lvert z \rvert <1}$ for the open disk. Let $\pi \colon \cX \to \Delta$ be a degeneration of hyper-Kähler varieties of dimension $2n$  over $\Delta^* =\Delta\setminus\{0\}$.
Consider the (log-)monodromy operator $N_i$ on the Betti cohomology $H^i(X,\Q)$ of a smooth fiber $X= \pi^{-1}(t)$ at $t \in \Delta^*$.
The monodromy operators satisfy
\[ \nu(N_{2i}) = i \nu(N_2) \] for every $0 \leq i \leq n$. Here $\nu(N)$ is the \emph{nilpotency index} of a nilpotent linear operator $N \in \End(V)$ on a finite-dimensional vector space over a field, i.e.,
\[
\nu(N) \coloneqq \min \{m \in \Z_{\geq 0} \mid N^{m+1} = 0\}.
\]
\end{conj}
The local monodromy theorem implies that $\nu(N_2) \in \Set*{0, 1,2}$ (cf.~\cite[(A.2)]{Landman}), and we say that the projective degeneration $\cX/\Delta$ is of Type I, II, or III respectively. \Cref{Nagai Conjecture over C} posits a deep relation between these three degeneration types and monodromy nilpotency indices in higher degrees, which has been studied by many people.
\begin{itemize}
    \item If $\cX/\Delta$ is of Type I ($\nu(N_2)=0$) or Type III ($\nu(N_2) =2)$, then Nagai's conjecture is confirmed by Kollár--Laza--Saccà--Voisin (cf.~ \cite[Theorem 1.7, Proposition 7.14 (2)]{KollarLaza2018a}). Soldatenkov \cite[Corollary 3.6]{Soldatenkov20a} gives a proof for possibly \emph{non-projective} degenerations of Type I.
    \item Nagai's conjecture is still widely open for Type II ($\nu(N_2)=1$) degenerations in general. Recently, Green--Kim--Laza--Robles \cite{GKLR} proved Conjecture \ref{Nagai Conjecture over C} when $X$ is in one of the four known deformation classes (i.e., $K3^{[n]}$-type, $\Kum_n$-type, $\OG_6$-type, and $\OG_{10}$-type). Their method involves a detailed study on the action of the \textit{Looijenga--Lunts--Verbitsky (LLV) Lie algebra} on the cohomology.
    \item In \cite{HuybrechtsMauri2023a}, Huybrechts--Mauri provide a different viewpoint for the Type II degenerations by the perverse filtration attached to an isotropic class.
\end{itemize}

It is generally believed that projective models over the ring of integers $\cO_K$ of $K$ are arithmetic analogues of ``degenerating families over the unit disk $\Delta$''; see \cite[Table 9.1]{DeligneHodgeI} for example. Thus, it is natural to ask an arithmetic analogue of \Cref{Nagai Conjecture over C}. More precisely, we formulate the following two conjectures.

If $\ell \neq p$, Grothendieck's $\ell$-adic monodromy theorem we mentioned before actually asserts that there is a nilpotent operator $N_{\ell,i} \in \End_{\Q_{\ell}}(H^i_{\et}(X_{\overline{K}},\Q_{\ell}))$ such that
$\phi_{\ell,i}(\sigma) = \exp(t_{\ell}(\sigma)N_{\ell,i})$
for any $\sigma \in I \subset I_K$,
where $t_{\ell} \colon I_K \to \Z_{\ell}$ is the composition of $\ell$-primary tame quotient of the inertia group with a fixed trivialization $\Z_{\ell}(1) \simeq \Z_{\ell}$.
\begin{conj}[$\ell$-adic analogue of Nagai's conjecture]
\label{ell adic Nagai Conjecture}For a hyper-Kähler variety $X$ over $K$ of dimension $2n$, the monodromy operators satisfy
$
\nu(N_{\ell,2i}) =  i\nu(N_{\ell,2})
$
for every $0 \leq i \leq n$.
\end{conj}

If $\ell =p$, by applying Fontaine's functor $\Dpst(-)$ 
for potentially semistable representations,
we obtain a $K_0^{\ur}$-vector space 
$
\Dpst(H^i_{\et}(X_{\overline{K}},\Q_p))
$
with a nilpotent operator
$N_{p,i}$, where $K_0^{\ur}$ is the maximal unramified extension of the fraction field $K_0$ of the ring of Witt vectors $W(k)$.
See \Cref{Subsection:Potentially semistable representations and p adic monodromy operators} below.

\begin{conj}[$p$-adic analogue of Nagai's conjecture]
\label{p adic Nagai Conjecture}
For a hyper-Kähler variety $X$ over $K$ of dimension $2n$, the monodromy operators satisfy 
$
\nu(N_{p, 2i}) =  i\nu(N_{p,2})
$
for every $0 \leq i \leq n$.
\end{conj}

We explore these two conjectures and, in particular, confirm their validity for hyper-Kähler varieties of four known deformation types.
\begin{thm}
\label{Main Theorem for known hk}
Let $X$ be a \hk variety over $K$ of dimension $2n$.
Assume that $X$ is one of four types: $K3^{[n]}$-type, $\Kum_n$-type, $\OG_6$-type, or $\OG_{10}$-type (see \Cref{Definition:known type hyper kahler}).
Then Conjecture \ref{ell adic Nagai Conjecture}
and 
Conjecture \ref{p adic Nagai Conjecture}
hold true for $X$.
\end{thm}

Furthermore, we have a geometric description for the monodromy nilpotency index $\nu(N_{\ell,2})$, which does not depend on $\ell$ (including $\ell = p$), in terms of the reduction type of \emph{Kuga--Satake abelian varieties} for any \hk variety $X$ over $K$ with $b_2(X) \geq 4$ (see \Cref{thm:l-adic monodromy}).
We thus obtain the notion of reduction type for \hk varieties $X$ over $K$; we say that $X$ has \emph{Type I, II, or III reduction} if $\nu(N_{\ell,2})= 0,1 \text{ or } 2$ respectively. %
By definition, we have $\nu_{\ell,i} = \nu(N_{\ell,i})$. We actually obtain a satisfactory answer to \Cref{Motivation} for hyper-Kähler varieties in these four types combining with \Cref{Main Theorem for known hk}.
A direct consequence of our discussion, together with Poincaré duality, is the following:

\begin{cor}
Assume $X$ satisfies conditions in \Cref{Main Theorem for known hk}. We have
$
\nu(N_{p,2i}) = \nu(N_{\ell,2i}) 
$
for any prime $\ell$ and any $0 \leq i \leq 2n$.
\end{cor}

\subsection{General remarks on arithmetic Nagai's conjectures}

As in \cite{GKLR}, the theory of LLV Lie algebras plays a crucial role in our proof.
Verbitsky’s description of the graded algebra generated by the second cohomology inside the full cohomology (often referred to as the Verbitsky component) implies that for any \hk variety $X$ over $K$ of dimension $2n$, 
the inequality
$\nu(N_{\ell, 2i}) \geq i \nu(N_{\ell, 2})$
holds for every $0 \leq i \leq n$ and $\ell$.
Since we have $\nu(N_{\ell, 2i}) \leq 2i$ in general, this implies \Cref{ell adic Nagai Conjecture}
and 
\Cref{p adic Nagai Conjecture} when $X$ has Type III reduction; see \Cref{thm:nagai for type III}.

A crucial step in the study of the case where $X$ has Type I or Type II reduction is to show that the $\ell$-adic and $p$-adic monodromy operators in all (even) degrees are contained in the LLV Lie algebra.
In fact, this inclusion is enough to prove the conjectures in the Type I case; see \Cref{thm:nagai for type I}.

For the four known deformation types, this can be deduced from the Mumford--Tate conjecture for the full cohomology; see \Cref{Section:Galois representations associated with hk varieties}.
On the other hand, for $p$-adic monodromy operators, we provide an alternative approach using Sen's theory as a new input.
This enables us to prove that $p$-adic monodromy operators lie in the LLV Lie algebra without using the Mumford--Tate conjecture, and even without assuming that our hyper-Kähler varieties belong to one of the four known deformation types.
See \Cref{Section:Sen theory and LLV Lie algebras} for details.

\begin{rem}\label{Remark:Griffiths transversality intro}
Along the way, we also present a Kuga--Satake type construction for higher-degree $p$-adic cohomology groups of \hk varieties.
As another application of our calculations, we will see that for any \hk variety $X$ over $K$ with $b_2(X) \geq 4$, the monodromy operator on $\Dpst(H^i_{\et}(X_{\overline{K}},\Q_p))$ satisfies \textit{Griffiths transversality} for all $i$.
See \Cref{Section: p-adic monodromy operators of hyper kahler varieties} for details.
\end{rem}

\begin{rem}
Unlike the $p$-adic version, we can prove $\ell$-adic Nagai's conjecture (\Cref{ell adic Nagai Conjecture}) of Type I only for the four known deformation types.
It is unclear how to show that the $\ell$-adic monodromy operators lie in the LLV Lie algebra in general.
\end{rem}

In the case of Type II reduction, 
we have to investigate the irreducible decomposition of the full cohomology with respect to the action of the LLV Lie algebra.
This decomposition is called the \textit{LLV decomposition}.
In \cite{GKLR}, Green--Kim--Laza--Robles gave a criterion for the validity of Nagai’s conjecture in terms of a representation-theoretic condition on the LLV decomposition.
Then they proved Nagai’s conjecture for the four known deformation types by verifying this representation-theoretic condition.
We employ the same strategy to establish the arithmetic analogues; see \Cref{thm:nagai for type II and known case}.

Finally, we remark that we can also compute nilpotency indices in odd degrees in some cases by a similar strategy; see \Cref{subsec:Nilpotency indices in odd degrees}.
For example, we can prove the following result, as observed by Soldatenkov in the topological case (cf.~\cite[Theorem 3.8]{Soldatenkov20a}).

\begin{thm}[\Cref{prop:type III odd cohomolgoy}]
\label{thm:type III odd cohomolgoy intro}
    Suppose $X$ is a hyper-Kähler variety over $K$ of dimension $2n$ with $b_2(X) \geq 4$ and $b_3(X) \neq 0$. Then we have
    $
    \nu(N_{p,2i+1}) \leq 2i-1
    $
    for any $1 \leq i \leq n-1$.
    If $X$ has Type III reduction, then the equality holds.
\end{thm}

\subsection{Organization of the paper}

In \Cref{Section:Preliminaries on Hk varieties}, we recall some basic definitions to 
fix our notation.
In \Cref{Section:Looijenga--Lunts--Verbitsky algebras}, we develop the theory of LLV Lie algebras for both $p$-adic and $\ell$-adic cohomological realizations, based on results known in the topological case.
In \Cref{Section:Galois representations associated with hk varieties}, we study the relation between Galois representations associated with étale cohomology of \hk varieties and the LLV Lie algebra, and explain that $\ell$-adic monodromy operators belong to the LLV Lie algebra for the four known deformation types.
We also review the Kuga--Satake construction and its relation to the LLV Lie algebra.
In \Cref{section: Reduction Types of Kuga--Satake}, we give a geometric description for the monodromy nilpotency index $\nu(N_{\ell,2})$ in terms of Kuga--Satake abelian varieties.
In \Cref{Section:Sen theory and LLV Lie algebras},
we study the $p$-adic LLV Lie algebra using Sen's theory, and prove that $p$-adic monodromy operators belong to the LLV Lie algebras.
Finally, in \Cref{sec:ArithmeticNagaiConj}, we prove the main results.

\section{Preliminaries}\label{Section:Preliminaries on Hk varieties}

\subsection{Hyper-Kähler varieties}\label{subsection:hyperkahler varieties}

We first recall the definition and basic properties of \hk varieties.
See, for example, \cite{HuybrechtsLecture}, \cite{Bindt}, and \cite{FLTZ22} for more details.

Let $X$ be a smooth projective variety over a field $L$ of characteristic zero.
We will always assume varieties are geometrically connected throughout this paper.
We say $X$ is a \emph{hyper-Kähler variety} if
$\pi_1^{\et} (X_{\overline{L}}) =1$
and $H^0 (X, \Omega^2_{X})=L \cdot  \omega$ for some $2$-form $\omega \colon \cO_X \to \Omega^2_{X}$ which is nowhere degenerate (i.e., the adjunction
$
T_{X} = \Omega_{X}^{1,\vee} \to \Omega_{X}^1
$
is an isomorphism).
The second condition implies the dimension of $X$ is even.
If there is a field embedding
$\sigma \colon L \hookrightarrow \C$,
then $X$ is a hyper-Kähler variety if and only if
the complex manifold
$X_{\sigma,\C}(\C)$
is a \hk manifold in the sense of \cite{HuybrechtsLecture}; see \cite[Lemma 3.1.3]{Bindt} and \cite[Section 2]{FLTZ22}.
Here we set $X_{\sigma,\C}\coloneqq X \otimes_{L,\sigma} \C$.

\begin{rem}
\label{rem:types}
Up to deformation equivalence, there exist the following four known types of complex \hk manifolds.
\begin{enumerate}
    \item{\bf($K3^{[n]}$-type):}
    Let $S$ be a complex algebraic K3 surface, and $n$ a positive integer.
    Then the Hilbert scheme of $n$-points $S^{[n]}$ is a \hk manifold of dimension $2n$.
    A \hk manifold $X$ is called \emph{$K3^{[n]}$-type} if $X$ is deformation equivalent to some $S^{[n]}$. We have $b_2(X) = 23$ if $n \geq 2$ and $b_2(X) = 22$ if $n = 1$. Moreover, $b_{2i+1}(X) =0$.
    \item{\bf($\Kum_{n}$-type):}
    Let $A$ be a complex algebraic abelian surface, %
    and $n\geq 2$ an integer.
    Let $A^{[n+1]}$ be the Hilbert scheme of $n$-points on $A$, and $K_{n}(A)$ the fiber of the summation morphism
    $A^{[n+1]} \rightarrow A$ over $0 \in A$.
    Then $K_{n} (A)$ is a \hk manifold of dimension $2n$.
    A \hk manifold $X$ is called \emph{$\Kum_{n}$-type} (or \emph{generalized Kummer type}) if $X$ is deformation equivalent to some $K_{n}(A)$.
    We have $b_2(X) = 7$ (as $n \geq 2$).
    \item{\bf($\OG_6$-type and $\OG_{10}$-type):}
    There exists a 6-dimensional (resp.~ 10-dimensional) exceptional \hk manifold constructed by O'Grady (\cite{OG6} (resp.~ \cite{OG10})).
    A \hk manifold $X$ is called \emph{$\OG_6$-type} (resp.~ \emph{$\OG_{10}$-type}) if $X$ is deformation equivalent to the O'Grady's example. We have $b_2(X) = 8$ (resp.~ $b_2(X) = 24$).
    Moreover, $b_{2i+1}(X) =0$ for these two deformation types.
\end{enumerate}
\end{rem}

\begin{lem}\label{lem:Deformation Types Stable Geometric Invariance}
	Let $X$ be a hyper-Kähler variety over $L$.
    If $X_{\sigma,\C}$ is of
    \begin{equation}\label{eq:KnownTypes}
    \tag{$\star$}
        K3^{[n]},~\Kum_{n},~\OG_6, \ \text{or } \OG_{10}
    \end{equation}
     types for some embedding $\sigma \colon L \hookrightarrow \C$, then so is $X_{\sigma',\C}$ for any embedding $\sigma' \colon L \hookrightarrow \C$.
\end{lem}
\begin{proof}
	See \cite[Proposition 2.3]{FLTZ22}.
\end{proof}

\begin{defn}\label{Definition:known type hyper kahler}
    Let $X$ be a hyper-Kähler variety over $L$.
    Let $L' \subset L$ be a subfield which is finitely generated over $\Q$ such that $X$ has a model $X'$ over $L'$.
    We say that $X$ is \textit{of a known deformation type \eqref{eq:KnownTypes}} if there exists an embedding $\sigma \colon L' \hookrightarrow \C$ such that 
    $X'_{\sigma,\C}$ is of a known deformation type as in \eqref{eq:KnownTypes}.
    By \Cref{lem:Deformation Types Stable Geometric Invariance}, this condition is independent of the choices of $L', X'$ and  $\sigma$.
\end{defn}

\begin{prop}
    Let $X$ be a \hk variety over $L$.
    Then there exists a unique quadratic form
    \[
    q \colon H^2_{\et} (X_{\overline{L}}, \widehat{\Z} (1)) \rightarrow \widehat{\Z}
    \]
    that is \emph{primitive} (i.e., if there is another $\widehat{\Z}$-quadratic form $q'$ such that $q = c q'$ with $c \in \widehat{\Z}$, then $c \in \widehat{\Z}^{\times}$) and is a $\Q$-multiple of the quadratic form
    \[
    \alpha \mapsto \int \sqrt{\mathrm{td}_{X_{\overline{L}}}} \alpha^2,
    \]
    such that for any ample line bundle $\mathscr{L}$ on $X$, we have $q (c_1(\mathscr{L})) \in \Z_{>0}$.
    Moreover $q$ is compatible with the action of the absolute Galois group
    $G_L\coloneqq \Gal (\overline{L}/L)$.
    Furthermore, if $L=\C$, $q$ comes from a primitive quadratic form 
    $q \colon {H^2 (X(\C),\Z)} \rightarrow \Z.
    $
    We call $q$ the  \emph{Beauville--Bogomolov--Fujiki (BBF) form} of $X$.
\end{prop}
\begin{proof}
See \cite[Lemma 4.2.1]{Bindt} (and also \cite[Theorem 4.2.4]{Bindt}).
\end{proof}

\subsection{$\ell$-adic monodromy operators}
\label{Subsection:ell adic monodromy operators}

In \Cref{Subsection:ell adic monodromy operators} and \Cref{Subsection:Potentially semistable representations and p adic monodromy operators}, we recall the definition of $\ell$-adic and $p$-adic monodromy operators, respectively.

For $\ell \neq p$, we fix an isomorphism $\Z_\ell(1) \simeq \Z_\ell$.
Let
$t_{\ell} \colon I_K \rightarrow \Z_{\ell} (1) \simeq \Z_\ell$
be the $\ell$-primary tame quotient of $I_K$, i.e.,  the homomorphism defined by $\sigma \mapsto (\sigma (\varpi^{1/\ell^m})/\varpi^{1/\ell^m})_m$
for a uniformizer $\varpi$ of $K$.

Let $X$ be a smooth proper scheme over $K$ of dimension $d$, and $i$ a non-negative integer.
Grothendieck's $\ell$-adic monodromy theorem (see \cite[Expos\'e I, Variante 1.3]{SGA7-I} and also \cite[Appendix]{Serre-Tate68}) implies that
there exists a unique nilpotent endomorphism
\[
N_{\ell, i} \colon H^i_\et(X_{\overline{K}}, \Q_\ell) \to H^i_\et(X_{\overline{K}}, \Q_\ell)
\]
such that for some open subgroup $I \subset I_K$, any $\sigma \in I$ acts on $H^i_\et(X_{\overline{K}}, \Q_\ell)$
as $\exp (t_{\ell} (\sigma) N_{\ell})$.
We call $N_{\ell, i}$ the \textit{($\ell$-adic) monodromy operator}.

We will need the following fact.

\begin{lem}\label{Lemma:ell adic nilpotency index upper bound in general}
    We have $\nu(N_{\ell, i}) \leq i$ for any integer $0 \leq i \leq d$.
\end{lem}

\begin{proof}
    Using de Jong's alteration \cite[Theorem 6.5]{deJong96}, we reduce to the case where $X$ has a strictly semistable model over $\mathcal{O}_K$.
    In this case, the assertion follows from \cite[Expos\'e I, Corollaire 3.4]{SGA7-I}.
\end{proof}

\subsection{Potentially semistable representations and $p$-adic monodromy operators}
\label{Subsection:Potentially semistable representations and p adic monodromy operators}

We recall Fontaine's formalism of linear-algebraic data for potentially semistable representations.
Throughout this paper, we will use the following notation.
Let $K_0$ be the fraction field of the ring of Witt vectors $W(k)$.
We may regard $K$ as a finite totally ramified extension of $K_0$.
Let $K_0^{\ur} \subset \overline{K}$ be the maximal unramified extension of $K_0$.
Let $B_{\dR}$, $B_{\cris}$ and $B_{\st}$ be Fontaine's period rings of $K$ {defined in \cite{Fontaine94II}}.

\begin{rem}\label{Remark: some choices for Bst}
As usual, we fix a valuation on $\overline{K}$ and an extended usual $p$-adic logarithm $\log \colon \overline{K}^\times \to \overline{K}$.
Following the constructions in 
\cite{Fontaine94II}, we have the corresponding
$B_{\cris}$-derivation $N \colon B_{\st} \to B_{\st}$, and an embedding $B_{\st} \hookrightarrow B_{\dR}$ over $B_{\cris}$.
\end{rem}

Let $V$ be a $p$-adic $G_K$-representation, that is, a finite-dimensional $\Q_p$-vector space with a continuous action of $G_K$.
(Here $G_K\coloneqq\Gal(\overline{K}/K)$.)
As in \cite{Fontaine94III},
we define
\[
    \Dpst(V) \coloneqq \varinjlim_{L/K}(V \otimes_{\Q_p} B_{\st})^{G_L},
\]
where $L$ runs over all finite extensions of $K$ contained in $\overline{K}$.
Then
$\Dpst(V)$ is a finite-dimensional $K_0^{\ur}$-vector space
with $\dim_{K_0^{\ur}} \Dpst(V) \leq \dim_{\Q_p} V$.
This is equipped with
the Frobenius map
$\varphi \colon \Dpst(V) \to \Dpst(V)$,
the monodromy operator
$N \colon \Dpst(V) \to \Dpst(V)$ (which is a $K_0^{\ur}$-linear map such that $N \varphi = p \varphi N$), and the Hodge filtration
$F^{\bullet}$
(which is a separated and exhaustive decreasing filtration on $\Dpst(V) \otimes_{K_0^{\ur}} \overline{K}$).
We say that $V$ is a \emph{potentially semistable} representation if $\dim_{K_0^{\ur}} \Dpst(V) = \dim_{\Q_p} V$.
In this case, we have a natural isomorphism
\[
\Dpst(V) \otimes_{K_0^{\ur}} B_{\st} \overset{\sim}{\to} V \otimes_{\Q_p} B_{\st}
\]
which is compatible with Frobenius maps and monodromy operators (in the usual sense).
If $V$ is potentially semistable, it is \emph{de Rham}  (by \cite[Th\'eor\`eme in \S 5.6.7]{Fontaine94III})
and we have 
\[
\Dpst(V) \otimes_{K_0^{\ur}} \overline{K} \overset{\sim}{\to} D_{\dR}(V) \otimes_K \overline{K}
\]
that is filtered.
Here $D_{\dR}(V)\coloneqq (V \otimes_{\Q_p} B_{\dR})^{G_K}$ is endowed with the natural filtration from $B_{\dR}$.

Let $X$ be a smooth proper scheme over $K$ of dimension $d$.
Then $H_{\et}^{i}(X_{\overline{K}},\Q_{p})$ is potentially semistable by the validity of the $C_{\st}$-conjecture (see \cite[Theorem 0.2]{Tsuji}) and de Jong's alteration \cite[Theorem 6.5]{deJong96}.
Moreover, the $j$-th successive quotient $F^{j}/F^{j+1}$
of the Hodge filtration $F^\bullet$ of $D_{\pst}(H_{\et}^{i}(X_{\overline{K}},\Q_{p}))$ is isomorphic to $H^{i-j}(X, \Omega^j_X) \otimes_K \overline{K}$.
The monodromy operator $N$ on $D_{\pst}(H_{\et}^{i}(X_{\overline{K}},\Q_{p}))$
will be denoted by $N_{p,i}$ and called the \textit{$p$-adic monodromy operator}.

As in \Cref{Lemma:ell adic nilpotency index upper bound in general}, we have an upper bound of the nilpotency index of $N_{p,i}$.
In fact, in the $p$-adic case, the following slightly more precise statement holds.

\begin{lem}\label{Lemma:p adic nilpotency index upper bound in general}
    Let $0 \leq i \leq d$.
    Let $j_1$ (resp.~ $j_2$) denote the largest (resp.~ smallest) integer $j$
    such that 
    $H^{i-j}(X, \Omega^{j}_X) \neq 0$.
    Then $\nu(N_{p,i}) \leq j_1 - j_2$.
    In particular, we have $\nu(N_{p,i}) \leq i$.
\end{lem}

\begin{proof}
We may assume that the residue field $k$ is algebraically closed. In this case, the underlying $F$-isocrystal $D_{\pst}(H^i_{\et}(X_{\overline{K}},\Q_p))$ admits the slope decomposition $\bigoplus_{\alpha \in \Q} D(\alpha)$ by Dieudonné--Manin's classification, where 
$D(\alpha)$ has a single slope $\alpha$.
We have $j_2 \leq \alpha \leq j_1$ for $D(\alpha)\neq 0$, by \cite[Proposition 5.4.2, Remark 4.4.6]{Fontaine94III}.
The $p$-adic monodromy operator satisfies $N_{p,i}\left(D(\alpha)\right) \subset D(\alpha-1)$ since $N_{p, i}\varphi=p \varphi N_{p, i}$, which implies $N^{j_1-j_2+1}_{p, i}=0$.
\end{proof}

\subsection{Representations of orthogonal Lie algebras}\label{Subsection:orthogonal Lie algebra}

Here we recall some facts and fix the notation on representations of orthogonal Lie algebras.
Our main references are \cite{FultonHarris} and \cite[Appendix A]{GKLR}.

Let $F$ be an algebraically closed field of characteristic zero.
Let $V$ be a non-degenerate quadratic space over $F$ of rank $N+2$ with $N>0$.
We consider the orthogonal Lie algebra
$\so(V)$ over $F$ and fix a Cartan subalgebra $\frh \subset \so(V)$.
Let us denote the non-trivial $\frh$-weights of the natural action of $\so(V)$ on 
$V$ by
\[
\{ \pm \epsilon_0, \pm \epsilon_1, \ldots, \pm \epsilon_r \} \subset \frh^*\setminus \{0 \}
\]
where $r = \lfloor \frac{N}{2} \rfloor$.
Then we consider the following positive system of roots $R^+$ and the corresponding set of dominant integral weights $\Lambda^+$:
\begin{itemize}[leftmargin=*]
    \item (Type $B_{r+1}$) If $N =2r+1$, then
    $
    R^+ \coloneqq \{ \epsilon_i + \epsilon_j \}_{i < j} \cup \{ \epsilon_i - \epsilon_j\}_{i < j} \cup \{ \epsilon_i \}_i
    $
    and the set of dominant integral weights is
    \begin{equation}\label{eq:DominantWeightB}
    \Lambda^+ = \Set*{ \mu =\sum_{i=0}^{r} \mu_i \epsilon_i \given \mu_0\geq \mu_1 \geq \cdots \geq \mu_{r} \geq 0, \, \mu_i \in \frac{1}{2} \Z, \, \mu_i-\mu_j \in \Z}.
    \end{equation}
    \item (Type $D_{r+1}$) If $N = 2r$, then $
    R^+ \coloneqq \{ \epsilon_i + \epsilon_j \}_{i < j} \cup \{ \epsilon_i - \epsilon_j\}_{i < j}
    $
    and the set of dominant integral weights is
    \begin{equation}\label{eq:DominantWeightD}
    \Lambda^+ \coloneqq \Set*{\mu =\sum_{i=0}^{r} \mu_i \epsilon_i \given \mu_0 \geq \mu_1 \geq \cdots \geq \mu_{r-1} \geq \vert\mu_r\vert, \, \mu_i \in \frac{1}{2} \Z, \, \mu_i-\mu_j \in \Z}.
    \end{equation}
\end{itemize}
For a dominant integral weight $\mu \in \Lambda^+$,
let $V_{\mu}$ denote the irreducible $\so(V)$-module with highest weight $\mu$.
We will identify $\Lambda^+$ with the set
    $S$
    of
    sequences
    $
    (\mu_0, \mu_1, \dotsc, \mu_r)
    $
    satisfying the conditions as in \eqref{eq:DominantWeightB} or \eqref{eq:DominantWeightD} depending on the parity of $N$.
    Then
    for a sequence
    $(\mu_0, \mu_1, \dotsc, \mu_r) \in S$,
    we have the corresponding irreducible $\so(V)$-module
    \[
    V_{\mu} = V_{(\mu_0, \mu_1, \dotsc, \mu_r)}.
    \]
\begin{rem}\label{Remark:prameter irreducible rep}
    We discuss how the construction
    \[
    (\mu_0, \mu_1, \dotsc, \mu_r) \mapsto V_{(\mu_0, \mu_1, \dotsc, \mu_r)}
    \]
    depends on the Cartan subalgebra $\frh \subset \so(V)$
    and the choices of the numbering and the sign for $\{ \pm \epsilon_i \}_{0 \leq i \leq r}$.
    Let 
    $
    V'_{(\mu_0, \mu_1, \dotsc, \mu_r)}
    $
    be the irreducible $\so(V)$-module corresponding to 
    a sequence
    $(\mu_0, \mu_1, \dotsc, \mu_r) \in S$ with respect to another choice of $\frh'$ and $\{ \pm \epsilon_i' \}_{0 \leq i \leq r}$.
    \begin{itemize}
        \item If $N$ is odd, then
    \[
    V_{(\mu_0, \mu_1, \dotsc, \mu_r)} \simeq V'_{(\mu_0, \mu_1, \dotsc, \mu_r)}
    \]
    for any
    $(\mu_0, \mu_1, \dotsc, \mu_r) \in S$.
    \item If $N$ is even,
    then
    for any
    $(\mu_0, \mu_1, \dotsc, \mu_r) \in S$,
    we have either
    \[
    V_{(\mu_0, \mu_1, \dotsc, \mu_r)} \simeq V'_{(\mu_0, \mu_1, \dotsc, \mu_r)} \quad \text{or} \quad V_{(\mu_0, \mu_1, \dotsc, \mu_r)} \simeq V'_{(\mu_0, \mu_1, \dotsc, -\mu_r)}.
    \]
    \end{itemize}
    To see this, since all Cartan subalgebras of $\so(V)$ are conjugate under inner automorphisms, 
    we may assume that $\frh=\frh'$.
    If $N$ is odd, the result follows from the fact that $\{ \pm \epsilon_i \}_{0 \leq i \leq r}$ and $\{ \pm \epsilon_i' \}_{0 \leq i \leq r}$ are conjugate under the action of the Weyl group.
    If $N$ is even, then $\{ \pm \epsilon_i' \}_{0 \leq i \leq r}$ is conjugate to 
    $\{ \pm \epsilon_0, \ldots, \pm \epsilon_{r-1}, \pm \epsilon_r \}$ or $\{ \pm \epsilon_0, \ldots, \pm \epsilon_{r-1}, \mp \epsilon_r \}$
    under the action of the Weyl group, which will give the same positive system $R^+$.
    These observations now lead to the result.
\end{rem}

\begin{rem}\label{rem:irreducible representations invariant under algebraically closed fields extension}
	Let $F \subset F'$ be an extension of algebraically closed fields.
    The construction
    $V_{(\mu_0,\mu_1,\cdots, \mu_r)}  \mapsto V_{(\mu_0,\mu_1,\cdots, \mu_r)} \otimes_F F'$
    induces a bijection between the set of irreducible representations of $\so(V)$ over $F$ and that of irreducible representations of $\so(V)_{F'}$ over $F'$.
    Indeed, by working with the base changes
    of
    $\frh$ and $\{ \pm \epsilon_i \}_{0 \leq i \leq r}$ to $F'$,
    we can parametrize irreducible representations of $\so(V)_{F'}$ over $F'$ by the same set $S$,
    and
    the irreducible representation of
    $\so(V)_{F'}$ corresponding to $\mu \in S$
    agrees with
    $V_{(\mu_0,\mu_1,\cdots, \mu_r)} \otimes_F F'$.
	\end{rem}

\section{Looijenga--Lunts--Verbitsky Lie algebras of varieties over $p$-adic fields}\label{Section:Looijenga--Lunts--Verbitsky algebras}

In this section, we develop the theory of 
the Looijenga--Lunts--Verbitsky Lie algebra (LLV Lie algebra) for $\ell$-adic and $p$-adic cohomology realizations.

\subsection{Looijenga--Lunts--Verbitsky algebras}\label{subsection:Looijenga--Lunts--Verbitsky algebras}
Let $E$ be a field of characteristic zero.
Let $d \geq 0$ be an integer.
Let $A^{\bullet} = \bigoplus_{0 \leq i \leq 2d} A^i$ be a graded $E$-algebra such that
$A^{\bullet}$ is finite-dimensional as an $E$-vector space and $\dim_E A^{2d} =1$.
We assume that $A^{\bullet}$ is a (graded) \emph{Frobenius algebra} (of degree $2d$),
i.e., the bilinear form
$A^{\bullet} \times A^{\bullet} \to A^{2d} \simeq E$
defined by $(x,y) \mapsto  (x\cdot y)_{2d}$ is non-degenerate.

For any $x \in A^2$,
let $e_x \colon A^{\bullet} \to A^{\bullet +2}$ be the map defined by
$a \mapsto x \cdot a$.
An element $x \in A^2$ satisfies the \emph{Hard Lefschetz (HL) property} if
$
(e_x)^i \colon A^{d-i} \to A^{d+i}
$
is bijective for any integer $0 \leq i \leq d$. On the other hand, we consider the shifted degree map
\[
h \colon A^{\bullet} \to A^{\bullet} \quad \text{such that} \quad h|_{A^{i}} = (i-d)\id_{A^i}
\]
for all $i$. By the Jacobson--Morozov theorem, $x\in A^2$ satisfies the HL property if and only if $(e_x,h)$ extends to an $\fsl_2$-triple, i.e., there is a map $f_x \in \End_{E}(A^{\bullet})$ such that
\[
[e_x, h] = -2e_x, \quad [f_x, h] = 2f_x, \quad [e_x,f_x] = h.
\]
Let $\mathfrak{a}_A \subset A^2$ be the set of all elements satisfying the HL property, which is a Zariski open subset.
We say that $A^{\bullet}$ is \emph{Lefschetz--Frobenius} if $\mathfrak{a}_A \neq 0$.
In this case, we define the \emph{total Lie algebra} of $A^{\bullet}$ as the Lie subalgebra
\[
\frg^{\tot}(A^{\bullet}) \subset \End_{E}(A^{\bullet})
\]
generated by all $\fsl_2$-triples $(e_x, h, f_x)$ for all $x \in \mathfrak{a}_{A}$.

\begin{lem}\label{lem:LLVundercomparison}
    Let $\gamma \colon A^{\bullet} \xrightarrow{\sim} B^{\bullet}$ be an isomorphism of Lefschetz--Frobenius $E$-algebras.
    The isomorphism
    $\End_E(A^{\bullet}) \xrightarrow{\sim} \End_E(B^{\bullet})$
    induced by $\gamma$
    maps $\frg^{\tot}(A^{\bullet})$ isomorphically onto $\frg^{\tot}(B^{\bullet})$.
\end{lem}

\begin{proof}
    This is clear from the definition.
\end{proof}

\begin{prop}\label{Proposition:scalar extension of gtot}
    Let $E \subset F$ be an inclusion of fields of characteristic zero. 
    Let $A^{\bullet}$ be a Lefschetz--Frobenius $E$-algebra
    and $A^{\bullet}_F\coloneqq A^{\bullet} \otimes_E F$.
    Then
    $
    \frg^{\tot}(A^{\bullet}) \otimes_{E} F = \frg^{\tot}(A_{F}^{\bullet})
    $
    as Lie subalgebras of $\End_{E}(A^{\bullet}) \otimes_E F=\End_{F}(A^{\bullet}_F)$.
\end{prop}
\begin{proof}
    Since $E$ is an infinite field,
    we see that
    $A^2 \subset A^2_F$ is Zariski dense.
    Thus 
    $\mathfrak{a}_{A} \subset \mathfrak{a}_{A_F}$ is also Zariski dense.
    Since 
    $\frg^{\tot}(A^{\bullet}) \otimes_{E} F$
    is Zariski closed in $\End_{F}(A^{\bullet}_F)$, it follows that $\frg^{\tot}(A^{\bullet}) \otimes_{E} F$
    contains $e_x$ and $f_x$ for all $x \in \mathfrak{a}_{A_F}$. This implies the assertion.
\end{proof}

Let $L$ be a field of characteristic zero.
Let $H^{\bullet}(-)$ be a Weil cohomology theory with coefficients in $E$
for smooth projective varieties $X$ over $L$, such that the hard Lefschetz theorem holds. Then, for any $X$ of dimension $d$, the cohomology algebra $H^{\bullet}(X)$ naturally forms a Lefschetz--Frobenius algebra of degree $2d$.
The \emph{LLV Lie algebra} of $X$ with respect to the cohomology theory $H^{\bullet}(-)$ is the total Lie algebra 
$\frg^{\tot}(H^{\bullet}(X))$.
Let
$\rho \colon \frg^{\tot}(H^{\bullet}(X)) \to \End_{E}(H^{\bullet}(X))$
denote the standard Lie algebra representation.

\begin{ex}\label{Example:Betti LLV}
    Assume that $L=\C$.
    We denote by $\frg_B(X)$ the LLV Lie algebra with respect to the Betti realization
    $H^{\bullet}_B(X)=H^{\bullet}(X(\C), \Q)$.
\end{ex}

\begin{ex}\label{Example:etale LLV}
    The LLV Lie algebra of $X$ with respect to the $\ell$-adic étale cohomology
    {$H^\bullet_\ell(X)=
    H_{\et}^{\bullet}(X_{\overline{L}},\Q_{\ell})$} is denoted by $\frg_\ell(X)$.
    There is a natural continuous $G_L$-action on $\frg_\ell(X)$.
    For an embedding $\overline{L} \hookrightarrow \C$, we have a natural identification of $\Q_{\ell}$-Lie algebras
    \begin{equation}\label{eq:BettiEtaleCompare}
    \frg_B(X  \otimes_L \C ) \otimes_{\Q} \Q_{\ell} \simeq \frg_\ell(X).
    \end{equation}
    by \Cref{Proposition:scalar extension of gtot}.
\end{ex}

Let the notation be as in Section \ref{Subsection:Potentially semistable representations and p adic monodromy operators}.
Let $A^\bullet$ be a Lefschetz--Frobenius algebra over $\Q_p$ with a continuous $G_K$-action that is compatible its algebra structure.
Assume that $A^\bullet$ is potentially semistable.
Since $\frg^{\tot}(A^{\bullet})$ is a $G_K$-subrepresentation of $\End_{\Q_p}(A^{\bullet})$, we obtain 
a Lie subalgebra
\[
\Dpst(\frg^{\tot}(A^{\bullet})) \subset \Dpst(\End_{\Q_p}(A^{\bullet}))=\End_{K_0^{\ur}}(\Dpst(A^{\bullet})).
\]
On the other hand
$\Dpst(A^{\bullet})= \bigoplus_i \Dpst(A^i)$
is naturally a Lefschetz--Frobenius algebra over
$K_0^{\ur}$
and thus we have the total Lie algebra
\[
\frg^{\tot}(\Dpst(A^{\bullet})) \subset \End_{K_0^{\ur}}(\Dpst(A^{\bullet}))
\]
of $\Dpst(A^{\bullet})$.

\begin{prop}\label{Proposition:LLV algebra for potentially semistable representation}
    If $A^\bullet$ is potentially semistable, then we have an identification
    \[
    \Dpst(\frg^{\tot}(A^{\bullet})) = \frg^{\tot}(\Dpst(A^{\bullet}))
    \]
    as Lie subalgebras of $\End_{K_0^{\ur}}(\Dpst(A^{\bullet}))$.
\end{prop}

\begin{proof}
    We have the inclusions
    $K_0^{\ur} \subset \overline{K} \subset B_{\dR}$
    of fields.
    It suffices to prove the equality after tensoring with $B_{\dR}$.
    Since $\End_{\Q_p}(A^{\bullet})$ is potentially semistable, so is $\frg^{\tot}(A^{\bullet})$.
    Thus we have the following identifications
    \[
    \Dpst(\frg^{\tot}(A^{\bullet})) \otimes_{K_0^{\ur}} B_{\dR} 
    = \frg^{\tot}(A^{\bullet}) \otimes_{\Q_p} B_{\dR}
    = \frg^{\tot}(A^{\bullet} \otimes_{\Q_p} B_{\dR})
    \]
    where the second equality follows from Proposition \ref{Proposition:scalar extension of gtot}.
    On the other hand,
    we have
    \[
    \frg^{\tot}(\Dpst(A^{\bullet})) \otimes_{K_0^{\ur}} B_{\dR} = \frg^{\tot}( \Dpst(A^{\bullet}) \otimes_{K_0^{\ur}} B_{\dR}) = \frg^{\tot}( A^{\bullet}\otimes_{\Q_p} B_{\dR})
    \]
    by Proposition \ref{Proposition:scalar extension of gtot} again.
    The assertion follows from these equalities. 
\end{proof}

\begin{ex}\label{Example:p-adic LLV}
Let $X$ be a smooth projective variety over $K$.
We
define
\[
\frg_{\pst}(X)\coloneqq \Dpst(\ge_p(X)) = \frg^{\tot}(\Dpst(H_{\et}^{\bullet}(X_{\overline{K}},\Q_{p}))),
\]
which is a Lie algebra over $K_0^{\ur}$ with a Frobenius map, a monodromy operator, and a filtration (on $\frg_{\pst}(X) \otimes_{K_0^{\ur}} \overline{K}$).
We have a natural isomorphism
of Lie algebras over $B_{\dR}$
\begin{equation}\label{eq:deRham comparison}
\frg_{\pst}(X)\otimes_{K_0^{\ur}} B_{\dR} \simeq \ge_p(X) \otimes_{\Q_p} B_{\dR}.
\end{equation}
\end{ex}

\subsection{LLV Lie algebras of hyper-Kähler varieties}\label{Subsection:LLV Lie algebras of hyper-Kähler varieties}
Here we record some important features of the LLV Lie algebras of hyper-Kähler varieties.

\begin{setup}\label{Definition:LLV algebra Betti etale p-adic notation}
Let us consider one of the following realization functors:
\begin{enumerate}[(a)]
    \item Let $X$ be a $2n$-dimensional \hk variety over $\C$.
    We set
    \[
    H^\bullet(X)\coloneqq H^\bullet_B(X)=H^\bullet(X(\C), \Q) \quad \text{and} \quad \ge(X)\coloneqq \ge_B(X)
    \]
    (see Example \ref{Example:Betti LLV}).
    Let $E\coloneqq  \Q$ and $F\coloneqq \C$.
    \item Let $X$ be a $2n$-dimensional \hk variety over a field $L$ of characteristic zero.
    We set
    \[
    H^\bullet(X)\coloneqq H^\bullet_\ell(X)=H^\bullet_{\et}(X_{\overline{L}}, \Q_\ell) \quad \text{and} \quad \ge(X)\coloneqq \ge_\ell(X)
    \]
    (see Example \ref{Example:etale LLV}).
    Let $E \coloneqq  \Q_\ell$ and $F \coloneqq \overline{\Q}_\ell$.
    \item 
    Let $X$ be a $2n$-dimensional \hk variety over $K$.
    We abuse notation slightly and write
    \[
    H^\bullet(X)\coloneqq H^\bullet_{\pst}(X)\coloneqq \Dpst(H_{\et}^{\bullet}(X_{\overline{K}},\Q_{p})) \quad \text{and} \quad \ge(X)\coloneqq \ge_{\pst}(X)
    \]
    (see Example \ref{Example:p-adic LLV}).
    Let $E \coloneqq  K^{\ur}_0$ and $F \coloneqq \overline{K}$.
    Note that, the Beauville--Bogomolov--Fujiki form $q$ on $H^2_{\et} (X_{\overline{K}}, \Q_p)$ induces a form on $H^2(X)$.
    We also write this form as $q$ and call it the Beauville--Bogomolov--Fujiki (BBF) form.
\end{enumerate}
In all cases,
$\ge(X) \subset \End_E(H^\bullet(X))$
is a Lie subalgebra over $E$ containing the degree operator $h$.
\end{setup}

\begin{rem}\label{Remark:Tate twist}
     Strictly speaking,
     the BBF form $q$ should be defined as a quadratic form on 
     $H^2(X)(1)\coloneqq H^2(X) \otimes_E E(1)$
     where $E(1)$ is the Tate twist, though one can choose an isomorphism
     $E(1) \simeq E$
     so that we can regard $q$ as a quadratic form on 
     $H^2(X)$.
     The resulting quadratic form on $H^2(X)$ is independent of the choice of $E(1) \simeq E$, up to scalar multiplication.
\end{rem}

We can describe the LLV Lie algebra $\ge(X)$ using the quadratic space
\[
\widetilde{H}^2 (X) \coloneqq  H^2 (X) \oplus U,
\]
where $U= E v \oplus Ew$ is equipped with the quadratic form $(a, b) \mapsto ab$.
The quadratic space
$\widetilde{H}^2 (X)$
is called the \textit{Mukai completion} of 
$H^2 (X)$.

\begin{rem}\label{Remark:Mukai completion as Frobenius algebra}
    We endow $\widetilde{H}^2 (X)$
    with the structure of a graded algebra
    $A^\bullet$ over $E$
    with grading given by
    \[
    A^0=Ev=E, \quad A^2=H^2 (X), \quad A^4=Ew=E
    \]
    and with multiplication determined by
    $x \cdot y = q(x, y)w$
    for all $x,y \in A^2=H^2 (X)$.
    Then $A^\bullet$ is a Lefschetz--Frobenius $E$-algebra of degree 4.
    By \cite[Theorem 9.1]{VerbitskyThesis},
    the total Lie algebra of 
    $A^\bullet$
    is $\so (\widetilde{H}^2 (X))$.
\end{rem}

\begin{thm}\label{Theorem:structure of LLV}
\label{defn-thm:llv}
In \Cref{Definition:LLV algebra Betti etale p-adic notation} (a), (b) or (c), the following assertions hold.
\begin{enumerate}
    \item
    The LLV Lie algebra $\ge(X)$ over $E$ is a semisimple Lie algebra.
    \item 
    There exists a decomposition 
    $
    \ge(X) = \ge(X)_2 \oplus \ge(X)_0 \oplus \ge(X)_{-2},
    $
    where $\ad h$ acts on $\ge(X)_{i}$ as multiplication by $i$.
    If we set $
    \overline{\ge}(X)\coloneqq [\ge(X)_0, \ge(X)_0],
    $
    then
    \[
    \ge(X)_0 = \overline{\ge}(X) \oplus E h.
    \]
    The subspace $\ge(X)_2$ (resp.~$\ge(X)_{-2}$) is generated by operators $e_x$ (resp.~$f_x$) for $x \in H^2(X)$ satisfying the HL property.
    \item
    The action of $\overline{\ge}(X)$ on
    $H^\bullet(X)$ preserves the grading.
    For $i \geq 0$, let
    \[
    \rho_i \colon \overline{\ge}(X) \rightarrow \End_E(H^i (X)) 
    \]
    be the natural representation.
    Then $\rho_{2}$ induces an isomorphism of Lie algebras over $E$
    \[
    \rho_2 \colon \overline{\ge}(X) \simeq \so (H^2 (X)).
    \]
    \item
    The isomorphism
    $
    \rho_2
    $
    extends uniquely to an isomorphism
    \[
    \psi \colon \ge(X) \simeq \so (\widetilde{H}^2 (X))
    \]
    of Lie algebras over $E$ with the following properties:
    \begin{itemize}
        \item $\psi(h)$ is zero on $H^2 (X)$ and we have $\psi(h)(v)=-2v$ and $\psi(h)(w)=2w$.
        \item For any $x \in H^2 (X)$,
        we have $\psi(e_x)(v)=x$, $\psi(e_x)(w)=0$, and $\psi(e_x)(y)=q(x, y)w$ for every $y \in H^2 (X)$.
    \end{itemize}
\end{enumerate}
\end{thm}

\begin{proof}
    For the Betti realization $\ge_B(X)$, the assertions follow from \cite{VerbitskyThesis},
    \cite{Looijenga-Lunts},
    \cite[Subsection 2.1.1]{GKLR}.
    See also \cite[Theorem 2.2]{Floccarigalois}.
    
    We shall prove (1)--(3) for realizations $\ge_\ell(X)$ and $\ge_{\pst}(X)$.
    For $\ge_\ell(X)$, by replacing $L$ by a subfield $L' \subset L$ which is finitely generated over $\Q$ such that $X$ is defined over $L'$, we may assume that there is an embedding $\overline{L} \hookrightarrow \C$.
    Then the assertions follow from those for $\ge_B(X_\C)$ and \eqref{eq:BettiEtaleCompare}.
    For $\ge_{\pst}(X)$,
    since (1)--(3) can be checked after taking the base change to $B_{\dR}$ (which is a field),
    these assertions follow from those for $\ge_p(X)$ and \eqref{eq:deRham comparison}.

    We shall prove (4).
    First, the uniqueness of $\psi$ follows from
    the fact that $\ge(X)_{-2}$ is generated by operators $f_x$ together with the Jacobson–Morozov theorem applying to the graded algebra $\widetilde{H}^2 (X)$ (see \Cref{Remark:Mukai completion as Frobenius algebra}).
    The existence of $\psi$
    for $\ge_\ell(X)$ follows easily from the case of $\ge_B(X)$ by the procedure as above.
    Let us prove the existence of $\psi$
    for $\ge_{\pst}(X)$.
    By the same argument as in the uniqueness part, we observe that if there exists
    an isomorphism as in the statement over $B_{\dR}$, then it automatically descends to an isomorphism over $K^{\ur}_0$.
    Thus the existence of $\psi$ follows from the case 
    of $\ge_p(X)$
    and \eqref{eq:deRham comparison}.
\end{proof}

\begin{defn}\label{Definition:reduced LLV}
    The Lie algebra $\overline{\ge}(X)$ as in Theorem \ref{Theorem:structure of LLV}
    is called the \emph{reduced LLV Lie algebra} of $X$ and is denoted by
    \[
    \overline{\ge}_B(X)
    \quad (\text{resp}.~ \overline{\ge}_\ell(X), \quad\text{resp}.~  \overline{\ge}_{\pst}(X))
    \]
    in the case (a) (resp.~ (b), resp.~ (c)).
\end{defn}

Let us discuss the representation-theoretic property of the action of $\ge(X)_F\coloneqq \ge(X) \otimes_E F$ on
$H^{\bullet}(X)_F\coloneqq H^{\bullet}(X) \otimes_E F$.
{We note that $b_2(X) \geq 3$, and thus we can apply the results of Section \ref{Subsection:orthogonal Lie algebra} to $\fg(X)_{F}$ (and $\overline{\fg}(X)_{F}$).}
We fix a Cartan subalgebra $\frh \subset \fg(X)_{F}$.
Let us denote the non-trivial $\frh$-weights of the natural action of $\frg(X)_F \simeq \so(\widetilde{H}^2 (X)_F)$ on the Mukai completion
$\widetilde{H}^2 (X)_F$ by
\[
\{ \pm \epsilon_0, \pm \epsilon_1, \ldots, \pm \epsilon_r \} \subset \frh^*\setminus \{0 \},
\]
where $r = \lfloor \frac{b_2(X)}{2} \rfloor$.
Then we choose a positive system of roots $R^+$ and the corresponding set of dominant integral weights $\Lambda^+$ of $\frg(X)_F$ as in Section \ref{Subsection:orthogonal Lie algebra}.
For a dominant integral weight $\mu =\sum_{i=0}^{r} \mu_i \epsilon_i \in \Lambda^+$,
let $V_{\mu, F}$ denote the irreducible $\frg(X)_F$-module with highest weight $\mu$.
We will identify $\mu$ with the sequence
$(\mu_0, \mu_1, \dotsc, \mu_r)$
and write
\[
V_{\mu,F}=V_{(\mu_0, \mu_1, \dotsc, \mu_r),F}.
\]
We will often omit the zero components of $\mu$ from the notation.
For example, we will write $(j)=(j, 0, \dotsc, 0)$ for $j \geq 0$.

\begin{defn}[LLV decomposition]\label{def:LLVdecomposition betti etale p-adic}
    The decomposition into irreducible $\frg(X)_F$-modules
    \[
    H^\bullet(X)_F = \bigoplus_{\mu \in \Lambda^+}V^{\oplus m_{\mu}}_{\mu,F}
    \]
    is called the \emph{LLV decomposition}, where $m_{\mu} \geq 0$.
\end{defn}

\begin{rem}\label{Remark:reduced LLV irreducible components}
We will also have to consider
the decomposition
of $H^i(X)_F$
into
irreducible
representations of the reduced LLV Lie algebra $\overline{\ge}(X)_F$.
For this, we will employ the following notation for irreducible $\overline{\ge}(X)_F$-modules.
    We fix a Cartan subalgebra
    $\overline{\frh} \subset \overline{\ge}(X)_F$
    and denote the non-trivial $\overline{\frh}$-weights of the natural action on
    $H^2(X)_F$ by
$
\{ \pm \epsilon'_1, \ldots, \pm \epsilon'_{r} \} \subset \overline{\frh}^*\setminus \{0 \}.
$
Let 
$\overline{\Lambda}^+$
be the set of sequences
$
\lambda=(\lambda_1, \lambda_2, \dotsc, \lambda_{r})
$
satisfying the conditions as in \eqref{eq:DominantWeightB} or \eqref{eq:DominantWeightD} depending on whether $b_2(X)$ is odd or even.
For each sequence
$
\lambda \in \overline{\Lambda}^+,
$
one can associate
an irreducible
$\overline{\ge}(X)_F$-module
\[
\overline{V}_{\lambda, F}=\overline{V}_{(\lambda_1, \lambda_2, \dotsc, \lambda_{r}), F}
\]
of highest weight $\lambda$.
\end{rem}

Let $SH^2 (X)  \subset H^{\bullet} (X)$ be the graded $E$-algebra generated by $H^2(X)$.
We call $SH^2 (X)$ the \emph{Verbitsky component} of $X$.
\begin{thm}
\label{thm:Verbitskycomponent Betti etale p-adic}
Recall that the dimension of $X$ is $2n$.
\begin{enumerate}
    \item The Verbitsky component $SH^2 (X)$ is a $\frg(X)$-submodule of $H^{\bullet} (X)$.
    \item Let $0 \leq i \leq n$. The map induced by the cup product 
    \[
    \Sym^{i}H^2 (X) \to H^{2i} (X)
    \]
    is injective. In other words, in the Verbitstky component, the degree $2i$ part is isomorphic to $\Sym^{i}H^2 (X)$.
\item $SH^2 (X)_F$ is an irreducible $\frg(X)_F$-module of highest weight $(n)$.
\end{enumerate}
\end{thm}

\begin{proof}
It is clear that
$SH^2 (X)$
is a
$\overline{\ge}(X)$-submodule of $H^{\bullet} (X)$.
We can check that 
$SH^2 (X)$
is also closed under the action of $h, \ge(X)_2, \ge(X)_{-2}$
by the same argument as in \cite[Theorem 2.15]{GKLR}.
This shows (1).

For $SH^2_B(X)$,
the assertions (2) and (3) are essentially obtained by
Verbitsky
as explained in \cite[Theorem 2.15]{GKLR}.
For $SH^2_\ell(X)$,
as in the proof of \Cref{defn-thm:llv}, we may assume that $\overline{L} \subset \C$.
Then the $\ge_\ell(X)$-module
$SH^2_\ell(X)$
can be identified with the base change of the $\ge_B(X_\C)$-module
$SH^2_B(X_\C)$.
Thus the assertions follow from
the case of $SH^2_B(X_\C)$
together with
\Cref{Remark:prameter irreducible rep}
and
\Cref{rem:irreducible representations invariant under algebraically closed fields extension}.

For $SH^2_{\pst}(X)$,
the assertions can be checked over
$B_{\dR}$ (or its algebraic closure) by the same remarks as above.
By \eqref{eq:deRham comparison},
after tensoring with $B_{\dR}$,
the $\ge_{\pst}(X)$-module
$SH^2_{\pst}(X)$
can be identified with the $\ge_p(X)$-module $SH^2_p(X)$.
Therefore
the results follow from the case of $SH^2_p(X)$.
\end{proof}

\begin{cor}
\label{cor:lowerboundbySym}
Let $N \in \overline{\frg}(X) \subset \End_E(H^{\bullet}(X))$ be a nilpotent operator. Denote $N_{i}$ for the restriction of $N$ on $H^i(X)$. Then we have
$
\nu(N_{2i}) \geq i \nu(N_2)
$
for any $0 \leq i \leq n$.
\end{cor}

\begin{proof}
For any $0 \leq i \leq n$, the restriction of $N_{2i}$ on $\Sym^i H^2(X) \subset H^{2i}(X)$ agrees with the action $\Sym^i N_2$ induced by $N_2$.
On the other hand, we have $\nu (\Sym^i N_2) =i \nu(N_2)$ by a direct computation (see \cite[Lemma 5.6]{GKLR}). Thus the required inequality follows.
\end{proof}

Similarly, we obtain the following result for Galois representations of $X$.
\begin{cor}\label{cor:lower bound of monodromy operators}
    Let $X$ be a \hk variety over $K$ of dimension $2n$.
    Let
$N_{\ell, 2i}$
be
the $\ell$-adic monodromy operator on $H^{2i}_{\et} (X_{\overline{K}}, \Q_{\ell})$
if $\ell \neq p$
and
let
$N_{p, 2i}$
be the $p$-adic monodromy operator on
$\Dpst(H^{2i}_{\et} (X_{\overline{K}}, \Q_{p}))$ if $\ell = p$.
    Then the inequality
$
\nu(N_{\ell, 2i}) \geq i \nu(N_{\ell, 2})
$
holds for all $0 \leq i \leq n$ and all $\ell$.
\end{cor}

\begin{proof}
    Since the injection
    $\Sym^{i}H^2(X) \to H^{2i}(X)$ 
    is compatible with ($\ell$-adic and $p$-adic) monodromy operators,
    this follows from the same argument as in \Cref{cor:lowerboundbySym}.
\end{proof}

For $X$ being a hyper-Kähler variety over $\C$ of one of the four known deformation types \eqref{eq:KnownTypes} as in Remark \ref{rem:types},
Green--Kim--Laza--Robles
in \cite{GKLR}
provide
a concrete description of the LLV decomposition of
$H^{\bullet}_B(X)_\C$
(or more strongly, the LLV decomposition of
$H^{\bullet}_B(X)_\R$).
As a consequence, we have the following result.

\begin{thm}[Green--Kim--Laza--Robles]
\label{thm:LLVdecompositionmu}
We assume that $X$ is in one of four known deformation types \eqref{eq:KnownTypes}.
For any $\mu= \sum_{i=0}^{r} \mu_i \epsilon_i \in \Lambda^+$ such that $V_{\mu, F}$ appears in $H^\bullet (X)_F$, we have
$
	\mu_0 +  \cdots + \mu_{r-1} + \vert \mu_{r} \vert \leq n.
$
In particular, 
$
\mu_0 + \mu_1 + \mu_2 \leq n.
$
\end{thm}

\begin{proof}
For $H^\bullet_B(X)$, this result is proved in \cite[Theorem 6.1]{GKLR}.
One can deduce the results for $H^\bullet_\ell(X)$ and $H^\bullet_{\pst}(X)$ from the case of $H^\bullet_B(X)$ by the same argument as in \Cref{thm:Verbitskycomponent Betti etale p-adic}.
(Again, we have used \Cref{Remark:prameter irreducible rep}
and
\Cref{rem:irreducible representations invariant under algebraically closed fields extension}.)
\end{proof}

We also include the following result.
Note that the action of $\ge(X)_F$ preserves
the decomposition
\[
H^\bullet(X)_F= H^{+}(X)_F \oplus H^-(X)_F,
\]
where $H^{+}(X)_F$ and $H^-(X)_F$ consist of parts of even degree and odd degree.

\begin{prop}\label{Proposition:reduced LLV irreducible component in general}
    \begin{enumerate}
        \item For every irreducible $\ge(X)_F$-module
    component
    $V_{\mu, F} \subset H^{+}(X)_F$, we have $\lambda_j \in \Z$ for any $j$. If $H^{-}(X)_F \neq 0$, then
    for every $V_{\mu, F} \subset H^{-}(X)_F$, we have $\lambda_j \in \frac{1}{2}\Z \backslash \Z$ for any $j$.
        \item For every irreducible $\overline{\ge}(X)_F$-module
    component
    $\overline{V}_{\lambda, F} \subset H^{2i}(X)$,
    we have $\lambda_j \in \Z$ for any $j$.
    If $H^{2i+1}(X)_F \neq0$, then for every
    $\overline{V}_{\lambda, F} \subset H^{2i+1}(X)$,
    we have $\lambda_j \in \frac{1}{2}\Z \backslash \Z$ for any $j$.
    \end{enumerate}
\end{prop}

\begin{proof}
    For $H^\bullet_B(X)$, these two statements are proved in \cite[Proposition 2.35]{GKLR}.
One can deduce them for $H^\bullet_\ell(X)$ and $H^\bullet_{\pst}(X)$ from the case of $H^\bullet_B(X)$, by the same argument as in \Cref{thm:Verbitskycomponent Betti etale p-adic}.
\end{proof}

\begin{cor}\label{cor:InjectiveLLV}
For $1 \leq  i \leq 4n-1$
    such that $H^i(X) \neq 0$,
    the homomorphism
    $\rho_{i} \colon \overline{\frg}(X) \to \End_E(H^{i}(X))$
    is injective.
In particular, this is the case when $i$ is even.
\end{cor}

\begin{proof}
    Note that $\overline{\frg}(X) \simeq \so(H^2(X))$ is a simple Lie algebra.
    Thus $\rho_{i}$ is injective if and only if the representation is non-trivial.
    If $1 \leq  i \leq 4n-1$ is even, then the latter holds
    by
    \Cref{thm:Verbitskycomponent Betti etale p-adic}
    (for $i \leq 2n$)
    and Poincar\'e duality (for $2n <i$).
    If $i$ is odd and $H^i(X) \neq 0$, then the latter holds by \Cref{Proposition:reduced LLV irreducible component in general}.
\end{proof}

For our purpose, it will be convenient to consider integrated representations
associated with LLV Lie algebras.
More precisely, we consider
the integration
\[
\widetilde{\rho} \colon \Spin(H^2(X)) \to \prod_i\GL_E(H^i(X))
\]
of the representation $\rho \colon \overline{\ge}(X) \to \prod_i\End_E(H^i(X))$.
Recall that 
\[
\GSpin(H^2(X)) = \Spin(H^2(X)) \cdot \Gm
\]
where $\Gm$ is the center of $\GSpin(H^2(X))$.
As discussed in
\cite[Section 6.2]{FFZ}
and
\cite[Section 2.1]{Floccarigalois},
we can extend $\widetilde{\rho}$ to a representation
\begin{equation}\label{Equation:twisted LLV representation}
    R \colon \GSpin(H^2(X)) \to \prod_i\GL_E(H^i(X))
\end{equation}
such that each $\lambda \in \G_m \subset \GSpin(H^2(X))$
acts on $H^i(X)$ as multiplication by $\lambda^i$ for all $i$.
This representation $R$ is called the
\textit{twisted LLV representation}.

\begin{rem}\label{Remark:twsited LLV representation and twisted degree operator}
The Lie algebra representation induced by $R$ is identified with the faithful representation
\begin{equation}\label{Equation:twisted g_0}
    \ge(X)^{\tw}_0 \coloneqq  \overline{\ge}(X) \oplus Eh^{\deg} \hookrightarrow \prod_i\End_E(H^i(X)),
\end{equation}
where $h^{\deg}$ is the map acting on each $H^i(X)$ as multiplication by $i$.
Note that we have
\[
h^{\deg}=h+2n.
\]
On the other hand, the LLV representation
$
\rho \colon \frg_0(X) \to \prod_i \End_E(H^i(X))
$
can be integrated to a representation
$
\GSpin(H^2(X)) \to \prod_i \GL_E(H^i(X))
$
extending 
$\widetilde{\rho}$
such that 
each $\lambda \in \G_m$ acts on $H^i(X)$ as multiplication by $\lambda^{i-2n}$ for all $i$.
\end{rem}

\subsection{Weil operator}
Here we recall Verbitsky's observation in {\cite{Verbitskyaction}}, which reveals that the LLV Lie algebra of a hyper-Kähler variety interestingly incorporates its Hodge structure information.

    Let $V$ be a
    $\Q$-Hodge structure
    (i.e., a direct sum of finitely many pure $\Q$-Hodge structures of 
possibly different weights).
    Let
    $
    V_{\C} = \bigoplus_{s,t \in \Z} V^{s,t}
    $
    be the Deligne splitting.
    The \textit{Weil operator} $W \in \End_\C (V_{\C})$ is the map defined by
    \[
    V^{s,t} \ni x \mapsto (t-s) \sqrt{-1} x \in  V^{s,t}.
    \]

\begin{thm}[Verbitsky]\label{thm:MTinsideLLV}
    Let $X$ be a \hk variety over $\C$.
    Then the Weil operator $W$
    of the full cohomology $H^{\bullet}_B(X)$
    is contained in $\overline{\frg}_B(X) \otimes_\Q \C$.
\end{thm}

\begin{proof}
This is observed by Verbitsky in \cite[Theorem 1.4]{Verbitskyaction}; see also \cite[Proposition 2.24]{GKLR}.
\end{proof}

\begin{rem}\label{rem:IsogenyMT} 
Let $\MT(X)$ be the Mumford--Tate group of the $\Q$-Hodge structure $H^{\bullet}_B(X)$.
Recall the twisted LLV representation
$
R \colon \GSpin(H^2_B(X)) \to \prod_i \GL_{\Q}(H^{i}_B(X))
$
from \eqref{Equation:twisted LLV representation}.
It follows from \Cref{thm:MTinsideLLV}
that $\MT(X)$ is contained in the image of $R$; see \cite[Lemma 6.7]{FFZ}.
\end{rem}

\section{Galois representations associated with \hk varieties}\label{Section:Galois representations associated with hk varieties}

Let $L$ be a field of characteristic zero and 
$X$ a \hk variety over $L$.
In this section, we recall the relation between the Galois representation
\[
\phi_X \colon G_L \to \prod_i \GL_{\Q_\ell}(H^i_\ell(X)),
\]
where $H^i_\ell(X) = H^i_\et(X_{\overline{L}}, \Q_\ell)$,
and the twisted LLV representation (see \eqref{Equation:twisted LLV representation})
\[
R \colon \GSpin(H^2_\ell(X)) \to \prod_i\GL_{\Q_\ell}(H^i_\ell(X)).
\]

\subsection{Galois representations for known deformation types}\label{Subsection:Galois representations for known deformation types}

If $\overline{L} \subset \C$,
the Artin comparison theorem provides a natural identification
\begin{equation}\label{eq:Artincomparison}
\prod_i \GL_{\Q_\ell}(H^i_\ell(X)) \simeq  \prod_i \GL_\Q(H^i_B(X_\C)) \times_{\Spec \Q} \Spec \Q_\ell.
\end{equation}
The following theorem is known as the \emph{Mumford--Tate conjecture} of hyper-Kähler varieties, which is confirmed by André (\cite{Andre96b}) in degree $2$ when $b_2(X) \geq 4$.
For the full cohomology, the validity of this conjecture for the four known deformation types is established in a series of works \cite[Theorem 1.1]{FloccariMumfordTate} and \cite[Theorem 1.18]{FFZ}.

\begin{thm}\label{Theorem:Mumford-Tate conjecture}
We assume that $L$ is finitely generated over $\Q$.
We fix an embedding $\overline{L} \hookrightarrow \C$.
Assume that $X$ is in one of four known deformation types \eqref{eq:KnownTypes}.
Let
$G^\circ_\ell(X)$ be the identity component of the Zariski closure of the image of $\phi_X$.
Then we have
\[
G^\circ_\ell(X) = \MT(X_\C) \times_{\Spec \Q} \Spec \Q_\ell
\]
under the identification \eqref{eq:Artincomparison}.
\end{thm}

Without the assumption that $L$ is finitely generated over $\Q$, we can still prove

\begin{cor}\label{Corollary:Galois in LLV for known types}
Assume that $X$ is in one of four known deformation types \eqref{eq:KnownTypes}.
There exists an open subgroup $H \subset G_L$
such that $\phi_X(H)$ is contained in
the image of the twisted LLV representation
$
R \colon \GSpin(H^2_\ell(X)) \to \prod_i\GL_{\Q_\ell}(H^i_\ell(X)).
$
\end{cor}

\begin{proof}
    By replacing $L$ by a subfield $L' \subset L$ which is finitely generated over $\Q$ such that $X$ is defined over $L'$, we may assume that we are in the situation of \Cref{Theorem:Mumford-Tate conjecture}.
    Then the assertion follows from \Cref{Theorem:Mumford-Tate conjecture} and
    \Cref{thm:MTinsideLLV} (see also \Cref{rem:IsogenyMT}).
\end{proof}

An important consequence is

\begin{cor}\label{Corollary:ell-adic monodromy operators in LLV for known types}
Let $X$ be a \hk variety over $K$ which is in one of four known deformation types \eqref{eq:KnownTypes}
and $\ell \neq p$.
    Then the $\ell$-adic monodromy operator
    \[
    N_\ell = (N_{\ell, i})_i \in \prod_i \End_{\Q_\ell}(H^i_{\ell}(X))
    \]
    is contained in the reduced LLV Lie algebra $\overline{\ge}_{\ell}(X)$.
\end{cor}

\begin{proof}
    It follows from \Cref{Corollary:Galois in LLV for known types} that
    $N_\ell \in \overline{\ge}_\ell(X) \oplus \Q_\ell h^{\deg} \subset \prod_i\End_{\Q_\ell}(H^i_\ell(X))$
    (cf.~ \eqref{Equation:twisted g_0}).
    Since $N_\ell$ is nilpotent, its trace is zero, and thus we have $N_\ell \in \overline{\ge}_\ell(X)$.
\end{proof}

\begin{rem}\label{Remark:p-adic monodromy operator in LLV will be proved later in general}
    \Cref{Corollary:Galois in LLV for known types} (together with the Kuga--Satake construction) also implies that the $p$-adic monodromy operator
$N_p=(N_{p,i})_i \in \prod_i \End_{K^{\ur}_0}(H^i_{\pst}(X))$
is contained in $\overline{\ge}_{\pst}(X)$.
The details will be given in Section \ref{Section:Sen theory and LLV Lie algebras}.
In fact, we will prove this by applying Sen's theory instead of \Cref{Theorem:Mumford-Tate conjecture},
which enables us to prove the same result without assuming that $X$ is of one of the four known deformation types \eqref{eq:KnownTypes}; see Theorem \ref{thm:universalmonodromyp-adic}.
\end{rem}

\subsection{Kuga--Satake constructions}\label{subsection:KugaSatakeconstruction}

We recall the Kuga--Satake construction for \hk varieties.
We assume that $b_2(X)\geq 4$.
Let $\scrL$ be an ample line bundle on $X$.

After replacing $L$ by its finite extension,
there exists
an abelian variety
$\KS(X,\scrL)$
over $L$
of dimension $2^{b_2(X)-2}$
with a $\Z/2\Z$-grading 
\[
\KS(X,\scrL)=\KS^+(X,\scrL) \times \KS^-(X,\scrL)
\]
satisfying the following properties.
Let
$
P^2_\ell (X)(1) \subset H^2_\ell(X)(1)
$
be the primitive part with respect to $\scrL$.
Let 
$\Cl \coloneqq \Cl (P^2_\ell (X)(1))$
be the (abstract) Clifford algebra over $\Q_\ell$.
Then $H^1_\ell(\KS(X,\scrL))$ admits a right action of $\Cl$ compatible with $\Z/2\Z$-gradings
and
the $G_L$-action on $H^1_\ell(\KS(X,\scrL))$, where we equip $\Cl$ with the trivial $G_L$-action.
Moreover, there exist $G_L$-equivariant isomorphisms
    \begin{equation}\label{equation:Kuga-Satake comparison isomorphism}
        \Cl  (P^2_{\ell} (X)(1))   \simeq \End_{\Cl} (H^1_{\ell} (\KS(X,\scrL))),
    \end{equation}
    \begin{equation}\label{equation:even Kuga-Satake comparison isomorphism}
        \Cl^+ (P^2_{\ell} (X)(1))   \simeq \End_{\Cl^+} (H^1_{\ell} (\KS^+(X,\scrL))).
    \end{equation}
Here the left hand sides are equipped with the $G_L$-action induced from that on $P^2_{\ell} (X)(1)$.
There exists an isomorphism
\begin{equation}\label{equation:Kuga-Satake trivialization}
    H^1_{\ell} (\KS(X,\scrL)) \simeq \Cl
\end{equation}
    compatible with $\Z/2\Z$-gradings and right actions of $\Cl$, such that via this isomorphism,
    the isomorphisms
    \eqref{equation:Kuga-Satake comparison isomorphism} and \eqref{equation:even Kuga-Satake comparison isomorphism} are given by left multiplication.
We call $\KS(X,\scrL)$ a \textit{Kuga-Satake abelian variety}.

\begin{rem}
The existence of such an abelian variety was originally observed by Deligne \cite{Deligne72} and André \cite{Andre96b}.
    One can also prove this result by using the period map over $\Q$ given in \cite{Bindt}, as in the case of $K3$ surfaces \cite{MadapusiPera}.
    Here, as in \cite{Charles13, MadapusiPera},
    we use 
    the full Clifford algebra rather than the even part.
\end{rem}

\begin{lem}\label{lem:Kuga-Satale for full cohomology}
\begin{enumerate}
    \item There exists a unique homomorphism
    $
    G_L \to \GSpin(P^2_{\ell} (X)(1))
    $
    which induces the usual actions of $G_L$ on $H^1_{\ell}(\KS(X,\scrL))$
    via \eqref{equation:Kuga-Satake comparison isomorphism} and on $P^2_{\ell} (X)(1)$.
\item 
There exist an isomorphism
\begin{equation}\label{eq:fullKugaSatakeAsCliffordAlgebra}
H^1_{\ell} (\KS(X,\scrL)^{\times 2}) \simeq \Cl(H^2_{\ell} (X)(1))
\end{equation}
of $\Q_\ell$-vector spaces
and a homomorphism
\[
\phi_{\KS} \colon G_L \to \GSpin(H^2_{\ell} (X)(1))
\]
which induces the usual actions of $G_L$ on 
$H^1_{\ell} (\KS(X,\scrL)^{\times 2})$
via \eqref{eq:fullKugaSatakeAsCliffordAlgebra}, and on $H^2_{\ell} (X)(1)$.
\end{enumerate}
\end{lem}

\begin{proof}
    (1) follows immediately.
    For (2), using the orthogonal decomposition
    $
    H^2_{\ell} (X)(1) = P^2_{\ell} (X)(1) \oplus \Q_\ell \cdot c_1(\scrL),
    $
    we obtain
    \[
    \Cl(H^2_{\ell} (X)(1)) = \Cl(P^2_{\ell} (X)(1)) \oplus \Cl(P^2_{\ell} (X)(1)) \cdot c_1(\scrL).
    \]
    In particular we have an inclusion
    $\GSpin(P^2_{\ell} (X)(1)) \hookrightarrow \GSpin(H^2_{\ell} (X)(1))$.
    Moreover,
    the isomorphism \eqref{equation:Kuga-Satake trivialization} induces \eqref{eq:fullKugaSatakeAsCliffordAlgebra}
    such that
    the composition
    \[
    G_L\to \GSpin(P^2_{\ell} (X)(1)) \hookrightarrow \GSpin(H^2_{\ell} (X)(1))
    \]
    induces the usual action of $G_L$ on $H^1_{\ell} (\KS(X,\scrL)^{\times 2})$.
    (See also \cite[Variant 4.1.3]{Andre96b}.)
\end{proof}

We consider the homomorphism $\phi_{\KS}$
constructed in \Cref{lem:Kuga-Satale for full cohomology}.
Here,
as in \Cref{Remark:Tate twist}, we identify $\GSpin(H^2_{\ell}(X)(1))$ with $\GSpin(H^2_{\ell}(X))$
by choosing an isomorphism $\Q_\ell(1) \simeq \Q_\ell$
(but we will distinguish the actions of $G_L$ on $H^2_{\ell}(X)(1)$ and $H^2_{\ell}(X)$).

\begin{prop}\label{Proposition:LLV and Kuga-Satake}
    Assume that the homomorphism $\phi_X \colon G_L \to \prod_i \GL_{\Q_\ell}(H^i_\ell(X))$ factors through $R(\GSpin(H^2_\ell(X))) \subset \prod_i \GL_{\Q_\ell}(H^i_\ell(X))$.
    Then, the composition
\[
G_L \xrightarrow{\phi_{\KS}} \GSpin(H^2_{\ell}(X)) 
\xrightarrow{R}  \prod_i\GL_{\Q_\ell}(H^i_\ell(X))
\]
agrees with the homomorphism $\phi_X$ after restricting to an open subgroup of $G_L$.
\end{prop}

\begin{proof}
By \cite[Lemma 6.5]{FFZ},
    the projection
    $\prod_i \GL_{\Q_\ell}(H^i_\ell(X)) \to \GL_{\Q_\ell}(H^2_\ell(X))$
    induces an isogeny with finite kernel from
    $R(\GSpin(H^2_{\ell}(X)))$
    onto its image.
    Therefore, it suffices to show that
    the action of $G_L$ on $H^2_\ell(X)$ induced by the composition
    \begin{equation}\label{equation:Kuga-Satake Galois action versus LLV}
G_L\xrightarrow{\phi_{\KS}} \GSpin(H^2_{\ell}(X)) \xrightarrow{R_2} \GL_{\Q_\ell}(H^2_\ell(X))
\end{equation}
is equal to the usual Galois action,
where $R_2$ is the composition of $R$ with the projection.
We remark that 
$R_2$ is not the same as the usual homomorphism
\[
\GSpin(H^2_{\ell}(X)) \to \SO(H^2_\ell(X)) \subset \GL_{\Q_\ell}(H^2_\ell(X)),
\]
but rather its twist by the spinor norm $\GSpin(H^2_{\ell}(X)) \to \G_m$.
By construction of 
$\KS(X,\scrL)$,
the composition
\[
G_L\xrightarrow{\phi_{\KS}} \GSpin(H^2_{\ell}(X)) \to \G_m
\]
agrees with the inverse of the cyclotomic character.
Thus, the action of $G_L$ induced by \eqref{equation:Kuga-Satake Galois action versus LLV} is the twist of the usual action on $H^2_{\ell}(X)(1)$ by the inverse of the cyclotomic character, that is $H^2_{\ell}(X)$.
\end{proof}

\begin{cor}\label{Corollary:split Galois equivariant embedding}
    We set $W\coloneqq H^1_{\ell} (\KS(X,\scrL)^{\times 2})$, viewed as $\GSpin(H^2_{\ell}(X))$-module by the identification \eqref{eq:fullKugaSatakeAsCliffordAlgebra}. 
    Let
    \[
\widetilde{\iota} \colon  H^\bullet_\ell(X) \hookrightarrow \bigoplus_{j \in J} W^{\otimes m_j} \otimes W^{\vee, \otimes m'_j}
\]
be a $\GSpin(H^2_{\ell}(X))$-equivariant embedding \footnote{Since $\GSpin(H^2_{\ell}(X))$ is reductive and its action on
$W$
is faithful, such an embedding always exists.}
for a finite set $J$ and a family of integers $\{m_j, m'_j\}_{j \in J}$.
Assume that $\phi_X$ factors through $R(\GSpin(H^2_\ell(X)))$.
Then there exists an open subgroup $H \subset G_L$
such that $\widetilde{\iota}$
is $H$-equivariant, and admits an $H$-equivariant splitting.
\end{cor}

\begin{proof}
This follows from Proposition \ref{Proposition:LLV and Kuga-Satake}.
For the second statement, we have used that
every finite-dimensional representation of $\GSpin(H^2_{\ell}(X))$ is semisimple.
\end{proof}

In the rest of this section, we assume that $L=K$ and $\ell=p$.
We shall recall a lemma concerning the compatibility between the 
$p$-adic monodromy operators and the Kuga--Satake construction.
This result will be used in the proof of \Cref{thm:universalmonodromyp-adic}.

For $W\coloneqq H^1_{p} (\KS(X,\scrL)^{\times 2}) \simeq \Cl(H^2_{p}(X)(1))$,
the natural embedding
\begin{equation}\label{Equation:Kuga-Satake embedding full cohomology}
    H^2_{p}(X)(1) \hookrightarrow \End_{\Q_p}(W)
\end{equation}
defined by left multiplication by elements in $H^2_{p}(X)(1)$ is
$\GSpin(H^2_{p}(X)(1))$-equivariant, and hence
$G_K$-equivariant.
We identify the Lie algebra of $\GSpin(H^2_{p}(X)(1))$ with the orthogonal Lie algebra
$\so(H^2_{p}(X)(1)) \oplus \Q_ph^{\deg}$
as in \eqref{Equation:twisted g_0}, and 
let
\[
    \rho^{\KS} \colon  \so(H^2_{p}(X)(1)) \oplus \Q_ph^{\deg} \hookrightarrow \End_{\Q_p}(W)
\]
be the induced Lie algebra representation, where $h^{\deg}$ acts as the identity
on
$W$.
Since $\rho^{\KS}$ is $G_K$-equivariant,
we then have the induced homomorphism
\[
\rho^{\KS} \colon \so(H^2_{\pst}(X)(1)) \oplus K^{\ur}_0h^{\deg} \hookrightarrow \End_{K^{\ur}_0}(\Dpst(W)),
\]
which we denote by the same symbol.

\begin{lem}\label{Lemma:nilpotent operator on Kuga-Satake}
    Let $N_{\KS} \in \End_{K^{\ur}_0}(\Dpst(W))$
    be the monodromy operator on
$\Dpst(W)$.
Then
$
N_{\KS} = \rho^{\KS} (N_{p, 2})
$
where $N_{p,2} \in \so(H^2_{\pst}(X)(1))$ is the monodromy operator on $H^2_{\pst}(X)(1)$.
\end{lem}
\begin{proof}
Let $\{ s_{\alpha}\}_{\alpha}\subset W^{\otimes}$ be a set of tensors defining $\GSpin(H^2_p(X)(1)) \subset \GL_{\Q_p}(W)$.
Here $W^{\otimes}$ is the direct sum of all the $\Q_p$-vector spaces which can be constructed from $W$ by taking duals and tensor products.
Then
$\so(H^2_{p}(X)(1)) \oplus \Q_ph^{\deg} \subset \End_{\Q_p}(W)$
agrees with the subspace consisting of all endomorphisms that annihilate $\{ s_{\alpha}\}_\alpha$.

Since each $s_{\alpha}$ is $G_K$-invariant, it induces a tensor
$s_{\alpha, \pst} \in \Dpst(W)^\otimes$ which is annihilated by the monodromy operator.
Since
$\so(H^2_{\pst}(X)(1)) \oplus K^{\ur}_0h^{\deg} \subset \End_{K^{\ur}_0}(\Dpst(W))$
agrees with the subspace consisting of all endomorphisms that annihilate $\{ s_{\alpha, \pst}\}_\alpha$, we have
$
N_{\KS} \in \so(H^2_{\pst}(X)(1)) \oplus K^{\ur}_0h^{\deg}.
$
Since $N_{\KS}$ is nilpotent,
it follows that $N_{\KS} \in \so(H^2_{\pst}(X)(1))$.
The homomorphism
\[
H^2_{\pst}(X)(1) \to \End_{K^{\ur}_0}(\Dpst(W))
\]
induced by \eqref{Equation:Kuga-Satake embedding full cohomology} is compatible with monodromy operators, which implies that
$N_{\KS} =N_{p,2}$ in $\so(H^2_{\pst}(X)(1))$.
\end{proof}

\section{Reduction types}
\label{section: Reduction Types of Kuga--Satake}

In this section, we will establish an arithmetic analogue of \cite{SS20}, which is essential in the proof of the main theorems.
In particular, we will show that there is a geometric description for the reduction types of the second cohomology of a hyper-Kähler variety over $K$.

\subsection{Reduction types of Kuga--Satake abelian varieties}\label{subsec:Reduction of Kuga-Satake} 

Let $X$ be a hyper-Kähler variety over $K$
and $\scrL$ an ample line bundle on $X$.
We assume that $b_{2} (X) \geq 4$.
By extending $K$ if necessary, we assume that the Kuga--Satake abelian variety $\KS(X,\scrL)$ and the additional structures on it as in Section \ref{subsection:KugaSatakeconstruction} are defined over $K$.
In particular, we have
the $\Z/2\Z$-grading
$\KS(X,\scrL)=\KS^+(X,\scrL) \times \KS^-(X,\scrL)$.

We set
\[
A\coloneqq \KS^+(X,\scrL) \quad \textup{(resp.~ } A\coloneqq \KS(X,\scrL))
\]
if $b_2 (X)$ is even (resp.~ odd).
Let $\ell$ be a prime.
If $\ell \neq p$, then let
$P^2(X)(1)$ be the primitive part of $H^2_{\et}(X_{\overline{K}}, \overline{\Q}_\ell)(1)$
with respect to $\mathscr{L}$
and
$H^1(A)\coloneqq H^1_{\et}(A_{\overline{K}}, \overline{\Q}_\ell)$.
If $\ell=p$, then
let
$P^2(X)(1)$ be the primitive part of 
$\Dpst(H_{\et}^{2}(X_{\overline{K}},\Q_{p})(1)) \otimes_{K^{\ur}_0} \overline{K}$
with respect to $\mathscr{L}$
and 
$H^1(A)\coloneqq \Dpst(H^1_{\et}(A_{\overline{K}}, \Q_p)) \otimes_{K^{\ur}_0} \overline{K}$.
Here we work over the algebraically closed fields
$F=\overline{\Q}_\ell$
or $F=\overline{K}$, respectively.

Let 
$M_\bullet(X)$ be the monodromy filtration on 
$P^2(X)(1)$
associated with the monodromy operator $N_2$ on $P^2(X)(1)$
as in \cite[Proposition 1.6.1]{WeilII}.
Similarly let
$M_\bullet(A)$
be the monodromy filtration on 
$H^1(A)$.
For an integer $i \in \Z$, let
\begin{align*}
    r_{i}(X) &\coloneqq  \dim_F \mathrm{gr}^M_i(P^2(X)(1)) = \dim_F M_{i}(X)/M_{i-1}(X), \\
    r_{i}(A) &\coloneqq  \dim_F \mathrm{gr}^M_i(H^1(A)) = \dim_F M_{i}(A)/M_{i-1}(A).
\end{align*}

\begin{rem}\label{Remark:range and symmetry}
By \Cref{Lemma:ell adic nilpotency index upper bound in general} and \Cref{Lemma:p adic nilpotency index upper bound in general},
$r_{i}(X)=0$ if $\vert i \vert > 2$ and $r_{i}(A)=0$ if $\vert i \vert > 1$.
By definition of monodromy filtrations, $r_{i}(X)=r_{-i}(X)$ and $r_{i}(A)=r_{-i}(A)$.
\end{rem}

\begin{thm}
\label{thm:l-adic monodromy}
In the above setting, assume that the abelian variety $A$ over $K$ has semi-abelian reduction.
\begin{enumerate}
    \item 
    If $\nu (N_{2})=0$,
    then we have
    \[
    r_2(X)=r_{-2}(X)=r_1(X)=r_{-1}(X)=0 \quad \text{ and } 
    \quad
    r_0(X)=b_2(X)-1.
    \]
    Also, we have $r_1(A)=r_{-1}(A)=0$ and $r_{0}(A) = b_1(A)$, and $A$ admits good reduction.
    \item 
    If $\nu (N_2) =1$, 
    then $b_2(X) \geq 5$, and we have 
    \[
    r_2(X)= r_{-2}(X)=0, \quad r_1(X) = r_{-1}(X)=2, \quad \text{ and } \quad  r_0(X)= b_2(X)-5.
    \]
     Moreover we have $\dim_{F} (\Im N_{2}) =2$.
    Also, we have $r_0(A)=b_1(A)/2,$ and
    $ r_{1}(A) = r_{-1}(A)=b_1(A)/4$, and the torus rank of a special fiber of a N\'{e}ron model of $A$ is a half of $\dim A$.
    \item
    If $\nu (N_{2}) =2$, 
    then we have
    \[
    r_2(X)= r_{-2}(X)=1, \quad r_1(X) = r_{-1}(X)=0, \quad \text{ and } \quad  r_0(X)= b_2(X)-3.
    \]
    Moreover we have $\dim_{F} (\Im N_{2}) =2$ and $\dim_{F} (\Im N^2_{2}) =1$.
    Also, we have 
    $ r_{1}(A) = r_{-1}(A)=b_1(A)/2$
    and
    $r_0(A)=0$, and $A$ admits totally toric reduction. 
\end{enumerate}
\end{thm}

\begin{proof}
To simplify the notation, we write $b\coloneqq b_2(X)$.
Let
$B$ denote the
even Clifford algebra
$\Cl^+(P^2(X)(1))$
if $b$ is even
and denote the
full Clifford algebra
$\Cl(P^2(X)(1))$
if $b$ is odd.
Then
there exist a right action of
$B$
on $H^1(A)$
which is compatible with the monodromy operator,
and an isomorphism
\[
\Cl^+(P^2(X)(1)) \simeq \End_{B} (H^1(A))
\quad (\text{resp.~ } \Cl(P^2(X)(1)) \simeq \End_{B} (H^1(A)))
\]
which is compatible with monodromy operators.
Here the monodromy operator on the left (resp.~ right) hand side is induced from that on $P^2(X)(1)$ (resp.~ $H^1(A)$).

Since the even (resp. full) Clifford algebra
$B$
is split as a central simple algebra over $F$,
there exists
a simple $B$-submodule
$H \subset H^1(A)$
which is preserved by the monodromy operator, and we have
an isomorphism
$
H^1(A) \simeq H^{\oplus m}
$
which is compatible with $B$-actions and monodromy operators.
It follows that
we have an isomorphism
\begin{equation}\label{eq:Morita equivalence with monodromy}
   \Cl^+(P^2(X)(1)) \simeq \End_{F} (H)
\quad (\text{resp.~ } \Cl(P^2(X)(1)) \simeq \End_{F} (H))
\end{equation}
which is compatible with monodromy operators.
We note that
$m=\dim_F(H)=2^{(b-2)/2}$ (resp.~ $m=\dim_F(H)=2^{(b-1)/2}$).

Let us compute the monodromy filtrations on both sides of \eqref{eq:Morita equivalence with monodromy}.
Let $du \colon \fsl_2 \to \so(P^2(X)(1))$ be a representation
with $du\begin{psmallmatrix}
0 & 0 \\
1 & 0 \\
\end{psmallmatrix}=N_2$ given by the Jacobson--Morozov theorem.
We consider the semisimple element
$\varphi\coloneqq \begin{psmallmatrix}
1 & 0 \\
0 & -1 \\
\end{psmallmatrix}$
and let $V_j \subset P^2(X)(1)$ be the subspace consisting of elements $v$ such that $\varphi(v)=jv$.
Then we have $M_i(X) = \bigoplus_{j \leq i} V_j$; see \cite[(1.6.8)]{WeilII}.
In particular we have $\gr^M_iP^2(X)(1) \simeq V_i$ and
$r_i(X)=\dim_F V_i$.
Let
\[
\{-2 \leq \beta_{1}\leq \cdots\leq \beta_{b-1} \leq 2\}
\]
be
the (ordered) multiset of eigenvalues of $\varphi$ acting on $P^2(X)(1)$ so that for any integer $-2 \leq i \leq 2$, the number of occurrences of $i$ in this multiset is equal to $r_i(X)$.
Let 
$\fsl_2 \to \End_F(\Cl^+(P^2(X)(1)))$
(resp.~ $\fsl_2 \to \End_F(\Cl(P^2(X)(1)))$)
be the homomorphism induced from 
$du \colon \fsl_2 \to \so(P^2(X)(1))$.
We see that the multiset of eigenvalues of $\varphi$ acting on $\Cl^{+}(P^2(X)(1))$ (resp.~ $\Cl(P^2(X)(1))$) is given by
\begin{equation}\label{equation:complete set of eigenvalues}
\begin{aligned}
\{ 0 \} \cup \{ \beta_{j_1}+ \cdots +\beta_{j_{2k}}\}_{k > 0, \,
1 \leq j_1< j_2 < \cdots < j_{2k} \leq b-1} \\
(\textup{resp.~ } \{ 0 \} \cup \{ \beta_{j_1} + \cdots +\beta_{j_{k}}\}_{k > 0, \, 1 \leq j_1 < j_2 < \cdots < j_{k} \leq b-1})
\end{aligned}
\end{equation}
via the identification
\[
\Cl^+(P^2(X)(1)) = \bigoplus_{k \geq 0}\bigwedge^{2k} P^2(X)(1) \quad (\text{resp.~ } \Cl(P^2(X)(1) = \bigoplus_{k \geq 0} \bigwedge^k P^2(X)(1)).
\]
By the same reasoning as above, we see that the dimension of the $i$-th successive quotient
\[
\mathrm{gr}^M_i(\Cl^+(P^2(X)(1))) \quad (\textup{resp.~ } \mathrm{gr}^M_i(\Cl(P^2(X)(1))))
\]
of the monodromy filtration of $\Cl^+(P^2(X)(1))$
(resp.~ $\Cl(P^2(X)(1))$) is equal to the number of occurrences of $i$ in the multiset \eqref{equation:complete set of eigenvalues}.

Using the same description of monodromy filtrations in terms of the Jacobson--Morozov theorem, one can show (see \cite[Proposition 1.6.9]{WeilII}) that
\[
\mathrm{gr}^M_i(\End_{F} (H)) \simeq \bigoplus_{i'+i''=i} \mathrm{gr}^M_{i'}(H) \otimes \mathrm{gr}^M_{i''}(H^\vee) \simeq \bigoplus_{i'+i''=i} \mathrm{gr}^M_{i'}(H) \otimes \mathrm{gr}^M_{-i''}(H)^\vee.
\]
If follows that
\[
\dim_F \mathrm{gr}^M_i(\End_{F} (H)) = \frac{1}{m^2}\sum_{i'+i''=i} r_{i'}(A) r_{i''}(A).
\]
In particular, we have $\mathrm{gr}^M_i(\End_{F} (H))=0$ if
$\vert i \vert > 2$.

By comparing these results via isomorphisms
\[
\mathrm{gr}^M_i(\Cl^+(P^2(X)(1))) \simeq \mathrm{gr}^M_i(\End_{F} (H)) \quad (\textup{resp.~ } \mathrm{gr}^M_i(\Cl(P^2(X)(1))) \simeq \mathrm{gr}^M_i(\End_{F} (H)))
\]
induced from \eqref{eq:Morita equivalence with monodromy},
we see that one of the following conditions holds:
\begin{enumerate}[(a)]
    \item
    $r_{2}(X) = r_{ -2}(X) = r_{1}(X)= r_{-1}(X)=0$ and $r_{ 0}(X) = b-1$.
    \item 
    $r_{2}(X)= r_{-2}(X) =0$, $r_{1}(X) = r_{-1}(X)=2$, and $r_{ 0}(X)= b-5$.
    \item 
    $r_{2}(X)= r_{-2}(X) =1$, $r_{1}(X) = r_{-1}(X)=0$, and $r_{ 0}(X)= b-3$.
    \item 
   $r_{2}(X)= r_{-2}(X) =0$, $r_{1}(X) = r_{-1}(X)=1$, and $r_{ 0}(X)= b-3$.
\end{enumerate}

In the case (a),
we have $\nu(N_{2}) =0$.
Since every eigenvalue in \eqref{equation:complete set of eigenvalues} is $0$, we have that $r_{1}(A) = r_{-1}(A) =0$. This means that $A$ admits good reduction
by \cite[Theorem 1]{Serre-Tate68} and \cite[Part II, Theorem 4.7]{Coleman-Iovita}.

Next, we treat the case (b).
Note that (b) does not occur when $b =4$.
In the following, we may assume that $b \geq 5$.
In this case, we have $\nu (N_{2}) = 1$ and $\dim_{F} (\Im N_{2}) =2$.
Since we have
\[
\dim_F \gr_2^M(\End_F(H))=r_1(A)^2/m^2
\]
and
\[
\dim_F \gr_2^M(\Cl^+(P^2(X)(1))) =2^{b-6} \quad (\text{resp.~ } \dim_F \gr_2^M(\Cl(P^2(X)(1)))=2^{b-5} ),
\]
it follows that
$
r_{1}(A) = 2^{(b-4)}
$
(resp.~ $r_{1}(A) = 2^{(b-3)}$), and thus $r_1(A)=b_1(A)/4$.
By \cite[Expos\'e IX]{SGA7-I} (especially \cite[Expos\'e IX, Proposition 3.5]{SGA7-I}) and \cite[Part II, Proposition 4.5]{Coleman-Iovita},
the torus rank of the special fiber of a N\'{e}ron model of $A$ is half of $\dim A$.

In the case (c),
we see that $\nu (N_{2}) = 2$, $\dim_{F} (\Im N_{2}) =2$, and
$\dim_{F} (\Im N^2_{2}) =1$.
Moreover, since eigenvalues in \eqref{equation:complete set of eigenvalues} 
(counting without multiplicity) are $\{-2,0,2\}$,
we have
\[
\dim_F \mathrm{gr}^M_1(\End_{F} (H))=0 \quad \text{and} \quad \dim_F \mathrm{gr}^M_2(\End_{F} (H)) \neq 0.
\]
It follows that $r_0(A)=0$.
Thus
$A$ admits totally toric reduction
(i.e., the torus rank of the special fiber of the Néron model of $A$ is $\dim A$) again by \cite[Expos\'e IX]{SGA7-I} and \cite[Part II, Proposition 4.5]{Coleman-Iovita}.

In the case (d),
we have that
$
\mathrm{gr}^M_2(\End_{F} (H))=0$
since eigenvalues in \eqref{equation:complete set of eigenvalues} are contained $\{-1,0,1\}$.
This implies that $r_{1}(A) =r_{-1}(A) =0$.
But then we have $\mathrm{gr}^M_1(\End_{F} (H)) = 0$, which contradicts that $1$ occurs at least once in \eqref{equation:complete set of eigenvalues} (where we use that $r_0(X)=b-3 > 0$).
Therefore, this case does not occur.
Since all cases are covered in (a)-(c), it finishes the proof.
\end{proof}

\begin{cor}
\label{cor:l-independence of nilpotency index}
For a hyper-Kähler variety $X$ over $K$ with $b_2(X) \geq 4$ and any $\ell \neq p$, we have
$
\nu(N_{\ell,2}) = \nu (N_{p,2}),
$
where $N_{\ell,2}$ (resp.~$N_{p,2}$) is the monodromy operator on 
$H^2_{\ell} (X)$ (resp.~$H^2_{\pst} (X)$).
\end{cor}
\begin{proof}
By extending $K$ if necessary, we may assume that the assumption of \Cref{thm:l-adic monodromy} is satisfied.
Then the result follows from \Cref{thm:l-adic monodromy} since the conditions on the reduction the Néron model are independent of $\ell$ and $p$.
\end{proof}

\subsection{Reduction types of hyper-Kähler varieties}\label{Subsection:Reduction types of hyper-Kähler varieties}

Let $X$ be a \hk variety over $K$. 
Let $N_{\ell,2}$ (resp.~ $N_{p,2}$) be the monodromy operator on 
$H^2_{\ell} (X)$ (resp.~ $H^2_{\pst} (X)$).

\begin{defn}\label{defn:Reduction Type at l}
For a prime number $\ell$ (including $\ell =p$), we say that $X$ has
\begin{itemize}
    \item Type I reduction if $\nu(N_{\ell,2}) =0$,
    \item Type II reduction if $\nu(N_{\ell,2}) =1$, or
    \item Type III reduction if $\nu(N_{\ell,2}) =2$.
\end{itemize}
By \Cref{cor:l-independence of nilpotency index} and \Cref{lem:nilpotencyH2 of b2=3} below, this definition does not depend on $\ell$.
\end{defn}

\begin{lem}\label{lem:nilpotencyH2 of b2=3}
 If $b_2(X) = 3$, we have $N_{\ell,2}=0$ for any $\ell$ (including $\ell =p$).
 \end{lem}
 \begin{proof}
For simplicity, we treat the $\ell \neq p$ case; the $p$-adic case follows by the same argument.
Let $\mathscr{L}$ be an ample line bundle on $X$.
Consider the orthogonal decomposition
$H^2_{\ell}(X)(1) = \Q_\ell \cdot c_1(\scrL) \oplus V$
where $V\coloneqq P^2_{\ell}(X)(1)$.
It suffices to show that $N_{\ell, 2} \vert_V =0$.
Since $\so(V)_{\overline{\Q}_\ell} \simeq \gl_1$ and $N_{\ell, 2} \vert_V \in \so(V)$ is nilpotent, it follows that $N_{\ell, 2} \vert_V=0$ as desired.
\end{proof}

The following fact allows us to describe the monodromy operators on a hyper-Kähler variety in a normalized basis.

\begin{prop}
\label{prop:conjugacyclassofN}
Let $X$ be a hyper-Kähler variety over $K$ and $\ell$ a prime number (including $\ell =p$).
Let $r \coloneqq\lfloor \frac{b_2(X)}{2} \rfloor$.
We put
\[
H^2(X) \coloneqq 
\begin{cases}
H^2_\ell(X) \otimes_{\Q_\ell} \overline{\Q}_{\ell}
& \textup{if }\ell \neq p, \\
H^2_{\pst} (X) \otimes_{K_{0}^{\mathrm{ur}}} \overline{K} & \textup{if }\ell = p.
\end{cases}
\]
The following statements hold.
\begin{enumerate}
    \item 
If $X$ has Type II reduction, then $b_2(X) \geq 5$, and there exists a basis 
\[
\{
e_1, \ldots, e_r, e_1', \ldots, e_r'
\}
\quad \mathrm{(resp.~ }\{e_1, \ldots, e_r, e_1', \ldots, e_r', e_{r+1}\})
\]
of $H^2(X)$
such that the matrix of $q$ is given by 
\begin{equation}
\label{eqn:qmatrix}
\begin{pmatrix}
    0 & \mathrm{id}_{r \times r}\\
    \mathrm{id}_{r \times r} & 0
\end{pmatrix}
\quad \textup{(resp.}
\begin{pmatrix}
    0 & \mathrm{id}_{r \times r} & 0 \\
    \mathrm{id}_{r \times r} & 0 &0\\
    0 & 0 & 1
\end{pmatrix}
),
\end{equation}
and 
\[
N_{\ell,2} ( \sum_{i=1}^{r\ (\textup{resp.~ $r+1$})} a_i e_i + \sum_{i=1}^{r} a_i' e_i')
= -a_2 e_1' + a_1e_2'
\]
if $b_2 (X)$ is even (resp.~ odd).
    \item
If $X$ has Type III reduction, 
then $b_2(X) \geq 4$, and there exists a basis 
\[
\{
e_1, \ldots, e_r, e_1', \ldots, e_r'
\}
\quad \mathrm{(resp.~ }\{e_1, \ldots, e_r, e_1', \ldots, e_r', e_{r+1}\})
\]
of $H^2(X)$
such that the matrix of $q$ is given by (\ref{eqn:qmatrix}) and
\[
N_{\ell,2} ( \sum_{i=1}^{r\ (\textup{resp.~ $r+1$})} a_i e_i + \sum_{i=1}^{r} a_i' e_i')
= a_1 e_2 - (a_2 +a_2') e_{1}' + a_1 e_2',
\]
if $b_2 (X)$ is even (resp.~ odd).
\end{enumerate}
\end{prop}

\begin{proof}
By \Cref{lem:nilpotencyH2 of b2=3}, we have $b_2(X) \geq 4$ in both cases.
If $X$ has Type II reduction, then $b_2(X) \geq 5$ by \Cref{thm:l-adic monodromy}.
For the existence of the required basis, the same proof as in \cite[Lemma 5.10]{GKLR} works by using \Cref{thm:l-adic monodromy}.
\end{proof}

\section{Sen theory and Looijenga--Lunts--Verbitsky Lie algebras}\label{Section:Sen theory and LLV Lie algebras}
For a $p$-adic Galois representation,
the Sen operator plays a role analogous to that of the Weil operator in Hodge theory.
In \Cref{Subsection:Sen operator}, we will review some facts on Sen's theory and prove that the Sen operator belongs to the LLV Lie algebra for \hk varieties.
Using this result, we prove in \Cref{Section: p-adic monodromy operators of hyper kahler varieties} that the $p$-adic monodromy operators are contained in the LLV Lie algebra.

\subsection{Sen operator}
\label{Subsection:Sen operator}

Let $C$ be the completion of $\overline{K}$.
Let 
$V$
be a $p$-adic $G_K$-representation over $\Q_p$.
We assume that $V$ is Hodge--Tate (for simplicity), so that we have a natural $G_K$-equivariant decomposition
\[
V_C = \bigoplus_{i \in \Z} D^i_{\mathrm{HT}}(V)\otimes_{K} C(-i),
\]
where
we set
$D^i_{\mathrm{HT}}(V)\coloneqq (V \otimes_{\Q_p} C(i))^{G_K}$
and $C(i)=C \otimes_{\Q_p} \Q_p(i)$ denotes the Tate twist.
The \textit{Sen operator}
$\Theta \colon V_C \to V_C$
is the $C$-linear homomorphism such that
$\Theta(x)=-ix$ for all $x \in D^i_{\mathrm{HT}}(V)\otimes_{K} C(-i)$.

\begin{thm}[Sen]\label{thm:Sen}
	Let
	$
	\phi \colon G_K \to \GL_{\Q_p}(V)
	$
    be a Hodge--Tate representation.
    Assume that the residue field $k$ is algebraically closed.
    \begin{enumerate}
        \item The Lie algebra $\ge$ of the $p$-adic Lie group  $\phi(G_K)$ is the minimal subalgebra of $\End_{\Q_p}(V)$ such that $\frg_{C}$ contains $\Theta$.
        \item There exists an open subgroup of $G_K$ that acts on $V \cap \ker \Theta$ trivially.
    \end{enumerate}
\end{thm}

\begin{proof}
    See \cite[Theorem 1]{Sen1973a} for (1).
    The assertion (2) follows by applying the first one to the Hodge--Tate representation $V \cap \ker \Theta$; see \cite[Corollary 1]{Sen1973a}.
\end{proof}

For a smooth proper variety $X$ over $K$,
there is
a $G_K$-equivariant decomposition
\[
H^i_p(X)_C \simeq \bigoplus_{s+t=i} H^t(X, \Omega^s_X) \otimes_K C(-s).
\]
In this paper,
we choose a decomposition induced by the comparison isomorphism given in
\cite[Theorem A1]{Tsuji02}.
We call it the \textit{Hodge--Tate decomposition}.
Via this decomposition,
the $(-s)$-eigenspace of the Sen operator $\Theta$ on $H^i_p(X)_C$ is $H^{t}(X, \Omega^s_{X})\otimes_K C(-s)$.

Recall the twisted LLV representation
$R \colon \GSpin(H^2_p(X)) \to \prod_i\GL_{\Q_p}(H^i_p(X))$
from \eqref{Equation:twisted LLV representation}
and the induced Lie subalgebra
\[
\ge_p(X)^{\tw}_0\coloneqq \overline{\ge}_p(X) \oplus \Q_ph^{\deg} \subset \prod_i\End_{\Q_p}(H^i_p(X))
\]
from \Cref{Remark:twsited LLV representation and twisted degree operator}.
Similar to the relation between the Weil operator and the LLV Lie algebra (\Cref{thm:MTinsideLLV}), we observe the following relation between the Sen operator and the LLV Lie algebra.
In fact, we deduce this result from \Cref{thm:MTinsideLLV}.

\begin{thm}\label{Theoremp:Sen is in g_0}
Let $X$ be a hyper-K\"ahler variety over $K$.
\begin{enumerate}
    \item The Sen operator
$\Theta \in \prod_i\End_{C}(H^i_p(X)_C)$
of
$H^\bullet_p(X)$ is contained in $\ge_p(X)^{\tw}_0 \otimes_{\Q_p} C$.
    \item There exists an open subgroup $I \subset I_K$ whose image
    under
    $G_K \to \prod_i\GL_{\Q_p}(H^i_p(X))$
    is contained in $R(\GSpin(H^2_p(X)))$.
\end{enumerate}
\end{thm}

\begin{proof}
(1)
Let $K' \subset K$ be a subfield which is finitely generated over $\Q$ such that
$X$ has a model $X'$ over $K'$.
We regard $K'$ as a subfield of $\C$ and 
let $Y\coloneqq X'_\C$.
We further choose a field
$L$ which sits in the following commutative diagram:
\[
\begin{tikzcd}[column sep = small]
    K'\ar[d] \ar[rr] & & \C \ar[d] \\
    K \ar[r] & C \ar[r] & L
\end{tikzcd}
\]
Let $W$ be the Weil operator on $H_{B}^\bullet(Y)_\C$.
We claim that there exists an isomorphism of graded algebras over $L$
\[
\gamma \colon
H^{\bullet}_{p}(X)_C \otimes_{C} L \simeq
H^{\bullet}_{B}(Y)_\C \otimes_\C L
\]
such that under this isomorphism, we have
\begin{equation}\label{equation:Sen operator and Weil operator}
    \Theta = \frac{1}{2\sqrt{-1}} W - \frac{1}{2}h^{\deg}.
\end{equation}
This claim, together with \Cref{thm:MTinsideLLV} and \Cref{lem:LLVundercomparison}, implies that
$\Theta \in \ge_p(X)^{\tw}_0 \otimes_{\Q_p} C$.

We shall prove the claim.
We fix an isomorphism $\Q_p(1) \simeq \Q_p$ of $\Q_p$-vector spaces.
Then, by \cite[Theorem A1, (A1.2)]{Tsuji02},
the Hodge--Tate decomposition induces isomorphisms of graded algebras over $L$:
\[
\bigoplus_i H^{i}_{p}(X)_C \otimes_{C} L \simeq \bigoplus_{i}\bigoplus_{s+t=i} H^t(X_L, \Omega^s_{X_L}) \otimes_{\Q_p} \Q_p(-s)
\simeq \bigoplus_{i}\bigoplus_{s+t=i} H^t(X_L, \Omega^s_{X_L}),
\]
where the algebra structure on the target is induced by the cup product on the Hodge cohomology.
On the other hand, by the Hodge decomposition,
we have the following isomorphisms of graded algebras over $L$:
\[
\bigoplus_{i}\bigoplus_{s+t=i} H^t(X_L, \Omega^s_{X_L})
\simeq
\bigoplus_{i}\bigoplus_{s+t=i} H^t(Y, \Omega^s_{Y}) \otimes_\C L
\simeq 
\bigoplus_{i} H^{i}_{B}(Y)_\C \otimes_\C L.
\]
By composing these isomorphisms, we obtain
$\gamma \colon
H^{\bullet}_{p}(X)_C \otimes_{C} L \simeq
H^{\bullet}_{B}(Y)_\C \otimes_\C L$.
With this construction, it is easy to check that the equality \eqref{equation:Sen operator and Weil operator} holds.

(2) This can be deduced from (1) and \Cref{thm:Sen} by the same argument as in \cite[Lemma 6.7]{FFZ}.
\end{proof}

\subsection{$p$-adic monodromy of \hk varieties}\label{Section: p-adic monodromy operators of hyper kahler varieties}

We are ready to show that
the monodromy operator
$N_p=(N_{p,i})_i \in \prod_i \End_{K^{\ur}_0}(H^i_{\pst}(X))$
is contained in the reduced LLV Lie algebra.

\begin{thm}	
\label{thm:universalmonodromyp-adic}
Let $X$ be a \hk variety over $K$.
\begin{enumerate}
    \item If $b_2(X) \geq 4$, then the monodromy operator
    $N_p \in \prod_i \End_{K^{\ur}_0}(H^i_{\pst}(X))$
of $H^\bullet_{\pst}(X)$ is contained in $\overline{\ge}_{\pst}(X)$.
    \item If $b_2(X) = 3$, then the monodromy operator
    $N_{p, 2i}$ of $H^{2i}_{\pst}(X)$ is zero for any integer $i$.
\end{enumerate}
\end{thm}
\begin{proof}
Without loss of generality, 
we may assume that the residue field $k$ is algebraically closed (cf.~ \cite[\S 5.1.5]{Fontaine94III}).
Assume that $b_2(X) \geq 4$.
We choose an ample line bundle $\scrL$ on $X$.
Then, after replacing $K$ by its finite extension, we can consider the Kuga--Satake abelian variety $\KS(X,\scrL)$ over $K$ and use the computations given in \Cref{subsection:KugaSatakeconstruction}.
We keep the notations there.

As in \Cref{Corollary:split Galois equivariant embedding},
for $W\coloneqq H^1_{p} (\KS(X,\scrL)^{\times 2})$,
there exists a $\GSpin(H^2_{p}(X))$-equivariant embedding
\[
\widetilde{\iota} \colon  H^\bullet_p(X) \hookrightarrow \bigoplus_{j \in J}W^{\otimes m_j} \otimes W^{\vee,\otimes m_j'}
\]
for a finite set $J$ and some integers $\{m_j, m'_j\}_{j \in J}$.
By
\Cref{Theoremp:Sen is in g_0}
and
\Cref{Corollary:split Galois equivariant embedding},
after replacing $K$ by its finite extension, the embedding $\widetilde{\iota}$ becomes $G_K$-equivariant.
Then we can apply $\Dpst$ to obtain
an embedding
\begin{equation}\label{eqn:pst Kuga Satake embedding for total cohomology}
\widetilde{\iota}_{\pst} \colon H^{\bullet}_{\pst} (X) \hookrightarrow \bigoplus_{j \in J}\Dpst(W)^{\otimes m_j} \otimes \Dpst(W)^{\vee,\otimes m'_j}.
\end{equation}
By construction, this is compatible with the action of
\[
\overline{\ge}_{\pst}(X)\simeq \so(H^2_{\pst}(X)) \simeq \so(H^2_{\pst}(X)(1)).
\]
Let $N_{\KS}$ be the monodromy operator of $\Dpst(W)$.
By \Cref{Lemma:nilpotent operator on Kuga-Satake}, we can view $N_{\KS} \in \overline{\ge}_{\pst}(X)$
(which agrees with the monodromy operator $N_{p,2}$ of $H^2_{\pst}(X)$).
Since 
$\widetilde{\iota}_{\pst}$
is also compatible with monodromy operators,
we see that $N_{\KS} \in \overline{\ge}_{\pst}(X)$ acts on $H^{\bullet}_{\pst} (X)$ as the monodromy operator $N_p$.
In particular we have $N_p \in \overline{\ge}_{\pst}(X)$, which proves (1).

One can argue for the even part as follows; this also applies to the case where $b_2(X)=3$.
Consider the integration
\[
\widetilde{\rho} \colon \Spin(H^2_p(X)) \to \prod_i\GL_{\Q_p}(H^i_p(X))
\]
of $\rho \colon \overline{\ge}_p(X) \to \prod_i\End_{\Q_p}(H^i_p(X))$.
The restriction of this representation to the even part $\bigoplus_i H^{2i}_p(X)$ factors through $\SO(H^2_p(X))$.
Therefore one can use the faithful $\SO(H^2_p(X))$-representation
$H^2_p(X)$ instead of $W$ to carry out the same argument as above to conclude that the monodromy operator
    $N^+_p \in \prod_i \End_{K^{\ur}_0}(H^{2i}_{\pst}(X))$
    of the even part
    is contained in
    (the image of)
    $\overline{\ge}_{\pst}(X)$.
In particular $N^+_p=0$ if $b_2(X)=3$ since in this case we have $N_{p, 2}=0$ by \Cref{lem:nilpotencyH2 of b2=3}.
This proves (2).
\end{proof}

We conclude this section with an application of our results.
For a potentially semistable
$G_K$-representation
$V$,
we say that $D_{\pst}(V)$
satisfies \emph{Griffiths transversality}
if $N(F^{\bullet}) \subset F^{\bullet-1}$
for its monodromy operator $N$ and Hodge filtration $F^\bullet$ (see \cite[Definition 1.3.1]{Paugam04a} for example).
A typical example of $D_{\pst}(V)$ which satisfies Griffiths transversality is 
$D_{\pst}(H^1_\et(A_{\overline{K}}, \Q_p))$
for an abelian variety $A$ over $K$ (in this case the transversality follows since $F^0=D_{\pst}(H^1_\et(A_{\overline{K}}, \Q_p))$ and $F^2=0$).
This, together with the results obtained above, allows us to prove the following.

\begin{cor}\label{cor:TransverseMonodromy}
Let $X$ be a hyper-Kähler variety over $K$ with $b_2(X) \geq 4$.
Then $H^i_{\pst}(X)=D_{\pst}(H^i_\et(X_{\overline{K}}, \Q_p))$ 
satisfies Griffiths transversality for all $i$.
\end{cor}
\begin{proof}
We may assume $k$ is algebraically closed.
Replacing $K$ by its finite extension,
we can take an embedding
\[
\widetilde{\iota}_{pst} \colon \bigoplus_i H^i_{\pst}(X) \hookrightarrow \bigoplus_{j \in J} H^1_{\pst} (\KS(X,\scrL)^{\times 2})^{\otimes m_j} \otimes H^1_{\pst} (\KS(X,\scrL)^{\times 2})^{\vee,\otimes m'_j}
\]
as in \eqref{eqn:pst Kuga Satake embedding for total cohomology}.
We recall that
this filtered embedding is strict; see \cite[Th\'eor\`eme 5.3.5]{Fontaine94III}.
Therefore, the assertion follows from the Griffiths transversality of $H^1_{\pst} (\KS(X,\scrL)^{\times 2})$.
\end{proof}

\section{Arithmetic analogue of Nagai's conjecture}\label{sec:ArithmeticNagaiConj}

In this section, we will prove our main results on the $\ell$-adic and $p$-adic analogues of Nagai's conjecture (\Cref{ell adic Nagai Conjecture} and \Cref{p adic Nagai Conjecture}).

 \subsection{A criterion of Green--Kim--Laza--Robles}\label{Subsection:Green--Kim--Laza--Robles criterion}

In \cite[Theorem 5.2]{GKLR}, it is shown that if $X$ is a complex hyper-Kähler variety of dimension $2n$ with $b_2(X) \geq 5$ such that the highest weight of any irreducible factor $V_{\mu, \C}$ in the LLV decomposition (\Cref{def:LLVdecomposition betti etale p-adic}) appearing in the even part $H^+_B(X)_\C=\bigoplus_i H^{2i}_B(X)_\C$
satisfies
the inequality 
\[
\mu_0 + \mu_1 + \mu_2 \leq n,
\]
then the Nagai conjecture holds for Type II degenerations of $X$.
In the following, we establish an analogue of their result in the arithmetic setting.

Let $X$ be a \hk variety over $K$ and $\ell$ a prime number (including $\ell=p$).
Following \Cref{Definition:LLV algebra Betti etale p-adic notation},
we write
\[
H^\bullet(X) \coloneqq  H^\bullet_\ell(X) \quad \text{and} \quad H^\bullet(X)_F\coloneqq  H^\bullet_\ell(X) \otimes_{\Q_\ell} \overline{\Q}_\ell \quad \text{if} \quad \ell \neq p
\]
and 
\[
H^\bullet(X) \coloneqq  H^\bullet_{\pst}(X) \quad \text{and} \quad H^\bullet(X)_F\coloneqq  H^\bullet_{\pst}(X)\otimes_{K_{0}^{\mathrm{ur}}} \overline{K} \quad \text{if} \quad \ell = p.
\]
Refer to \Cref{def:LLVdecomposition betti etale p-adic} for the LLV decomposition of $X$, i.e., the decomposition into irreducible $\ge(X)_F$-modules:
\[
H^\bullet(X)_F = \bigoplus_{\mu \in \Lambda^+} V_{\mu, F}^{\oplus m_{\mu}}.
\]
(Here, we have fixed a Cartan subalgebra $\frh \subset \ge(X)_F$ and a positive system of roots as in \Cref{Subsection:LLV Lie algebras of hyper-Kähler varieties}.
    These choices do not affect the conclusions below by \Cref{Remark:prameter irreducible rep}.)

Let $N_{\ell, 2} \in \overline{\ge}(X) \simeq \so(H^2(X))$ be the monodromy operator on $H^2(X)$.
Let
\[
N'_{\ell, i} \colon H^i(X)_F \to H^i(X)_F
\]
be the image of
$N_{\ell, 2}$ under the projection
$\rho_i \colon \overline{\ge}(X)_F \to \End_F(H^i(X))$.
See \Cref{Subsection:LLV Lie algebras of hyper-Kähler varieties} for the notation used here.

\begin{thm}\label{thm:arithmeticGreenKimetc}
    Let $X$ be a hyper-Kähler variety of dimension $2n$ over $K$ and $\ell$ a prime number.
    We assume that $X$ has Type II reduction.
    Then the following conditions are equivalent.
    \begin{enumerate}
        \item We have $\nu(N'_{\ell, 2i}) =i$ for all $0 \leq i \leq n$.
        \item For all irreducible factors $V_{\mu, F} \neq 0$ in the LLV decomposition which are contained in the even part
        $
        H^+(X)_F=\bigoplus_i H^{2i}(X)_F,
        $
        the following inequality holds:
        \[
	\mu_0 + \mu_1 + \mu_2 \leq n.
\]
    \end{enumerate}
\end{thm}

With \Cref{prop:conjugacyclassofN}, the same argument as given in \cite{GKLR} implies the theorem.
Here, we will provide a slightly different proof using the branching rule of successive restrictions of representations of orthogonal Lie algebras (cf.~ \cite[Section 25.3]{FultonHarris} and \cite[Appendix B.2]{GKLR}).

\begin{rem}
    If $X$ has Type II reduction, then $b_2(X) \geq 5$ by \Cref{prop:conjugacyclassofN}.
    In the rest of this subsection, we will always assume that $b_2(X) \geq 5$.
\end{rem}

Recall that $h \in \ge(X)$ is the shifted degree map which acts on $H^i(X)_F$ multiplication by $i-2n$.
In the following lemma, we will use the notation of \Cref{Remark:reduced LLV irreducible components}.

\begin{lem}\label{Lemma:restriction of LLV factors to reduced LLV}
    We assume that $b_2(X) \geq 5$.
    Let $r = \lfloor \frac{b_2(X)}{2} \rfloor$.
    Let $V_{\mu, F}$ be the irreducible $\ge(X)_F$-module corresponding to a dominant integral weight $\mu=(\mu_0, \mu_1, \dotsc, \mu_r) \in \Lambda^+$.
    \begin{enumerate}
        \item
        Let $V_{\mu, F}(-2\mu_0)$ be the $(-2\mu_0)$-eigenspace of $h \in \ge(X)$.
    If $b_2(X)$ is odd, then the $\overline{\ge}(X)_F$-module $V_{\mu, F}(-2\mu_0)$ contains the irreducible $\overline{\ge}(X)_F$-module
    $\overline{V}_{(\mu_1, \mu_2, \dotsc, \mu_r), F}$
    with highest weight $(\mu_1, \mu_2, \dotsc, \mu_r) \in \overline{\Lambda}^+$.
    If $b_2(X)$ is even, then $V_{\mu, F}(-2\mu_0)$ contains at least one of $\overline{V}_{(\mu_1, \mu_2, \dotsc, \mu_r), F}$
    or
    $\overline{V}_{(\mu_1, \mu_2, \dotsc, -\mu_r), F}$.
        \item 
        Assume that $V_{\mu, F}$ is contained in $H^\bullet(X)_F$.
        For an irreducible $\overline{\ge}(X)_F$-submodule $\overline{V}_{\lambda, F} \subset V_{\mu, F}$ (with $\lambda=(\lambda_1, \dotsc, \lambda_r) \in \overline{\Lambda}^+$) which is contained in $H^i(X)_F$ for some $0 \leq i \leq 2n$,
        we have the following inequality:
        \[
        \mu_0 + \mu_1 + \mu_2  \geq \lambda_1 + \lambda_2 + n - i/2.
        \]
    \end{enumerate}
\end{lem}
\begin{proof}
    This is well-known and follows easily from the branching rule.
    We provide a direct proof for the convenience of the reader.
    We may assume that Cartan subalgebras $\frh \subset \ge(X)_F$ and $\overline{\frh} \subset \overline{\ge}(X)_F$ are given as follows.
    (Note that  $\mu_0, \mu_1, \mu_2, \lambda_1, \lambda_2$ are non-negative since $b_2(X) \geq 5$, and will not change if we change Cartan subalgebras and positive systems of roots by Remark \ref{Remark:prameter irreducible rep}.)
    Recall
    $
    \psi_F \colon \ge(X)_F \simeq \so (\widetilde{H}^2 (X)_F)
    $
    from \Cref{Theorem:structure of LLV}.
    Here $\widetilde{H}^2 (X)_F = H^2(X)_F \oplus Fv \oplus Fw$ is the Mukai completion.
    If $b_2(X)$ is odd (resp.~ even), we can choose a basis 
\[
\{v_0, \ldots, v_r, w_0, \ldots, w_r, v_{r+1}\}
\quad \mathrm{(resp.~ } \{
v_0, \ldots, v_r, w_0, \ldots, w_r
\})
\]
of $\widetilde{H}^2 (X)_F$
such that $v_0=v$, $w_0=w$ and other vectors $v_i, w_i$ are contained in $H^2(X)_F$,
and the associated matrix of the quadratic form is given as in \eqref{eqn:qmatrix}.
By the standard process of constructing Cartan subalgebras in terms of this basis (see \cite[Lecture 18]{FultonHarris} or \cite[Section 21 (j)]{Milne17} for example), 
we can choose a Cartan subalgebra $\overline{\frh} \subset \overline{\ge}(X)_F$ such that
\begin{itemize}
    \item the elements
$v_1, \ldots, v_r, w_1, \ldots, w_r$ are weight vectors of the non-trivial $\overline{\frh}$-weights
$\epsilon'_1, \ldots, \epsilon'_r, -\epsilon'_1, \ldots, -\epsilon'_r$ of $H^2(X)_F$, respectively, and
    \item $\frh\coloneqq Fh \oplus \overline{\frh}$ is a Cartan subalgebra of $\ge(X)_F$ and we can write the non-trivial $\frh$-weights
$\{\pm \epsilon_m\}_{0 \leq m \leq r}$
of $\widetilde{H}^2(X)_F$ such that under the identification $\frh^*=(Fh)^* \oplus \overline{\frh}^*$, we have $\pm \epsilon_0 \in (Fh)^*$ with $\epsilon_0(h)=-2$ and $\epsilon_m = \epsilon'_m \in \overline{\frh}^*$ for all $1 \leq m \leq r$.
\end{itemize}
We then fix positive systems of roots as in \Cref{Subsection:orthogonal Lie algebra}.

The highest weight vector $v_\mu \in V_{\mu, F}$ is contained in the $(-2\mu_0)$-eigenspace $V_{\mu, F}(-2\mu_0)$ of $h \in \ge(X)$ and
the $\overline{\ge}(X)_F$-submodule of $V_{\mu, F}(-2\mu_0)$ generated by $v_\mu$ is an irreducible $\overline{\ge}(X)_F$-module with highest weight $(\mu_1, \dotsc, \mu_r) \in \overline{\Lambda}^+$.
This proves (1).

Now we prove (2).
The shifted degree operator $h$ acts on 
$\overline{V}_{\lambda, F}$ as multiplication by $i-2n$, since $\overline{V}_{\lambda, F} \subset H^i(X)_F$.
Let $\overline{v}_\lambda \in \overline{V}_{\lambda, F}$ be the highest weight vector.
Then $\overline{v}_\lambda$, viewed as an element of $V_{\mu, F}$, is a weight vector for the $\frh$-weight
\[
  (n-i/2) \epsilon_0 + \lambda_1\epsilon_1 + \lambda_2\epsilon_2 + \cdots +\lambda_r\epsilon_r.
\]
Since $\mu= 
\mu_0\epsilon_0 + \mu_1\epsilon_1 + \mu_2\epsilon_2 + \cdots +\mu_r\epsilon_r$ is the highest  weight, we obtain the inequality
$\mu_0 + \mu_1 + \mu_2  \geq \lambda_1 + \lambda_2 + n - i/2$.
\end{proof}

\begin{lem}\label{Lemma:nilpotency index for irreducible factors}
    We assume that $X$ has Type II reduction.
    Let $\lambda = (\lambda_1, \lambda_2, \dotsc, \lambda_r) \in \overline{\Lambda}^+$ be a dominant integral weight with $\lambda_i \in \Z$ for all $1 \leq i \leq r$ and let
    $\rho \colon \overline{\ge}(X)_F \to \End_F(\overline{V}_{\lambda, F})$
    be the corresponding irreducible $\overline{\ge}(X)_F$-module with highest weight $\lambda$.
    Then we have $\nu(\rho(N_{\ell, 2}))= \lambda_1 + \lambda_2$.
\end{lem}

\begin{proof}
As $X$ has Type II reduction,
    by using the normalized basis given in \Cref{prop:conjugacyclassofN} (1), we can prove this lemma by the same argument as in \cite[Lemma 5.8]{GKLR}.
\end{proof}

\begin{proof}[Proof of \Cref{thm:arithmeticGreenKimetc}]

$(1) \Rightarrow (2)$: 
We argue by contradiction and assume that there is some
irreducible $\ge(X)_F$-submodule
$V_{\mu, F} \subset H^+(X)_F$ such that
$\mu_0 + \mu_1 + \mu_2 >n$.
Since the $(-2\mu_0)$-eigenspace $V_{\mu, F}(-2\mu_0)$ of $h$ is contained in $H^{2n-2\mu_0}(X)_F$,
it follows from \Cref{Lemma:restriction of LLV factors to reduced LLV}
that 
$H^{2n-2\mu_0}(X)_F$
contains the irreducible $\overline{\ge}(X)_F$-module
$\overline{V}_{\lambda, F}$
where $\lambda=(\mu_1, \dotsc, \mu_r)$ if $b_2(X)$ is odd, and 
$\lambda=(\mu_1, \dotsc, \mu_r)$
or $\lambda=(\mu_1, \dotsc, -\mu_r)$ if $b_2(X)$ is even.
We note that $\mu_i \in \Z$ by \Cref{Proposition:reduced LLV irreducible component in general}.
It then follows from
\Cref{Lemma:nilpotency index for irreducible factors} that
\[
n-\mu_0=\nu(N'_{\ell, 2n-2\mu_0}) \geq \nu(N'_{\ell, 2n-2\mu_0}\vert_{\overline{V}_{\lambda, F}})=\mu_1 + \mu_2 > n-\mu_0,
\]
which is a contradiction.

$(2) \Rightarrow (1)$: 
We fix an integer $0 \leq i \leq n$.
By \Cref{cor:lowerboundbySym}, we have
$\nu(N'_{\ell, 2i}) \geq i$.
In order to show that $\nu(N'_{\ell, 2i}) =i$, it suffices to show 
the inequality
$
\nu(N'_{\ell, 2i}\vert_{\overline{V}_{\lambda, F}}) \leq i
$
for every irreducible $\overline{\ge}(X)_F$-module component
$\overline{V}_{\lambda, F} \subset H^{2i}(X)_F$.
Let us choose an irreducible $\ge(X)_F$-module component
$V_{\mu, F} \subset H^{+}(X)_F$ which contains $\overline{V}_{\lambda, F}$.
Then, by our assumption and \Cref{Lemma:restriction of LLV factors to reduced LLV},
we have 
\[
n \geq \mu_0 + \mu_1 + \mu_2 \geq \lambda_1 + \lambda_2 + n -i.
\]
This, together with \Cref{Lemma:nilpotency index for irreducible factors} (and \Cref{Proposition:reduced LLV irreducible component in general}), implies 
$
\nu(N'_{\ell, 2i}\vert_{\overline{V}_{\lambda, F}}) = \lambda_1 + \lambda_2 \leq i
$
as required.
\end{proof}

\subsection{Arithmetic analogue of Nagai's conjecture}\label{subsection:Arithmetic analogue of Nagai's conjecture}

Now we can prove our main results.
Let $X$ be a $2n$-dimensional \hk variety over $K$.
Let
$N_{\ell, i}$
be
the $\ell$-adic monodromy operator on $H^i_\ell(X)=H^{i}_{\et} (X_{\overline{K}}, \Q_{\ell})$
if $\ell \neq p$
and
let
$N_{p, i}$
be the $p$-adic monodromy operator on
$H^i_{\pst}(X)=\Dpst(H^{i}_{\et} (X_{\overline{K}}, \Q_{p}))$ if $\ell = p$.
\subsubsection*{Proof of \Cref{Main Theorem for known hk}}
The proof will be divided into three cases based on the reduction type.  We start with the Type III case.
\begin{thm}[Type III]
\label{thm:nagai for type III}
    Assume that $X$ has Type III reduction.
    Then we have
$
\nu (N_{\ell,2i}) =2i
$
for all $0 \leq i \leq n$ and for all $\ell$.
\end{thm}

\begin{proof}
    This follows from 
    \Cref{Lemma:ell adic nilpotency index upper bound in general}, \Cref{Lemma:p adic nilpotency index upper bound in general}, and \Cref{cor:lower bound of monodromy operators}.
\end{proof}

For the Type I case, we provide a slightly expanded version.

\begin{thm}[Type I]
\label{thm:nagai for type I}
Let $1 \leq i \leq 4n-1$ be an integer such that $b_i(X) \neq 0$.
If $b_2(X) = 3$, we further assume that $i$ is even.
\begin{enumerate}
    \item Then $N_{p, 2}=0$ if and only if $N_{p, i}=0$.
    \item Let $\ell \neq p$ be a prime number. Assume that $X$ is in one of four known deformation types \eqref{eq:KnownTypes}.
    Then $N_{\ell, 2}=0$ if and only if $N_{\ell, i}=0$.
\end{enumerate}
\end{thm}

\begin{proof}
    We have $N_{p, i}=\rho_{i}(N_{p, 2})$ by \Cref{thm:universalmonodromyp-adic}.
    The assertion (1) follows from this equality, together with the injectivity of $\rho_i$ (see \Cref{cor:InjectiveLLV}).
    The assertion (2) follows from the same argument using \Cref{Corollary:ell-adic monodromy operators in LLV for known types}.
\end{proof}

Finally, we treat the Type II case.
\begin{thm}[Type II]
\label{thm:nagai for type II and known case}
Assume that $X$ has Type II reduction.
We suppose that $X$ is in one of the four known deformation types \eqref{eq:KnownTypes}.
    Then we have
    $\nu (N_{\ell,2i}) = i$ for all $0 \leq i \leq n$ and for all prime numbers $\ell$.
\end{thm}

\begin{proof}
   By \Cref{Corollary:ell-adic monodromy operators in LLV for known types} and  \Cref{thm:universalmonodromyp-adic}, the monodromy operator $N_{\ell, 2i}$ agrees with $N'_{\ell, 2i}=\rho_{2i}(N_{\ell, 2})$. Thus, the result now follows from \Cref{thm:LLVdecompositionmu} and \Cref{thm:arithmeticGreenKimetc}.
\end{proof}

Now, \Cref{Main Theorem for known hk} can be concluded from \Cref{thm:nagai for type III}, \Cref{thm:nagai for type I}, and \Cref{thm:nagai for type II and known case}. Furthermore, in the $\ell=p$ case, even if $X$ is not known to be in one of the four deformation types \eqref{eq:KnownTypes}, we still have the following consequence.

\begin{thm}
\label{thm:nagai for type II and general case}
Assume that $X$ has Type II reduction.
Then the equality $\nu (N_{p,2i}) = i$ holds for all $0 \leq i \leq n$ if and only if the inequality $\mu_0 + \mu_1 + \mu_2 \leq n$ holds for every irreducible factor $V_{(\mu_0, \mu_1, \dotsc, \mu_r), \overline{K}} \neq 0$ in the LLV decomposition of $H^+_{\pst}(X)_{\overline{K}}$.
\end{thm}

\begin{proof}
    By \Cref{thm:universalmonodromyp-adic}, the $p$-adic monodromy operator $N_{p, 2i}$ agrees with $N'_{p, 2i}=\rho_{2i}(N_{p, 2})$.
    So this is a consequence of \Cref{thm:arithmeticGreenKimetc}.
\end{proof}

\subsection{Nilpotency indices in odd degrees}\label{subsec:Nilpotency indices in odd degrees}

In this part, we want to make further remarks on the nilpotency indices of the monodromy operators on cohomology in odd degrees and their relation to those on the second cohomology.
In the four known deformation types \eqref{eq:KnownTypes},  $\Kum_n(n \geq2)$ is the sole example that has non-trivial odd cohomology groups (see \Cref{rem:types}).

In \cite[Proposition 3.15]{Soldatenkov20a} and \cite[Proposition 3.12]{HuybrechtsMauri22}, it is shown that, for a Type III projective one-parameter degeneration $\cX/\Delta$ of a \hk manifold $\cX_0$ of dimension $2n$ with $b_3(\cX_0) \neq 0$, we have $\nu(N_{2i+1}) = 2i-1$ for the monodromy operator $N_{2i+1}$ on  $H^{2i+1}_B(\cX_0)$ for any $1 \leq i \leq n-1$.

Let $X$ be a hyper-Kähler variety over $K$ of dimension $2n$.
Similarly, we obtain the following description of monodromy nilpotency indices for
$H^{2i+1}_{\pst}(X)$:

\begin{thm}\label{prop:type III odd cohomolgoy}
Suppose that $b_2(X) \geq 4$.
\begin{enumerate}
    \item We have $\nu(N_{p,2i+1}) \leq 2i-1$
for all $1 \leq i \leq n-1$.
    \item If $X$ has Type III reduction and $b_3(X) \neq 0$, then $\nu(N_{p,2i+1}) = 2i-1$
for all $1 \leq i \leq n-1$.
\end{enumerate}
\end{thm}
\begin{proof}
    Since 
    $H^{2i+1}(X,\cO_{X})=H^0(X,\Omega^{2i+1}_{X})=0$ (cf.~ \cite[Proposition 3]{Beauville83}),
    (1) follows from \Cref{Lemma:p adic nilpotency index upper bound in general}.
    Note that we have the normalized basis for $H^2_{\pst}(X)\otimes_{K^{\ur}_0} \overline{K}$ as in \Cref{prop:conjugacyclassofN}.
    Then the claim (2) follows from the same computation as in \cite[Proposition 3.15]{Soldatenkov20a}.
\end{proof}

\begin{cor}\label{cor:nilpotency of H^3 p-adic}
    Suppose that $b_2(X) \geq 4$ and $b_3(X) \neq 0$.
    We have
    $\nu(N_{p, 3})=0$ if $X$ is of type I, and $\nu(N_{p, 3})=1$ if $X$ is of type II or III.
\end{cor}

\begin{proof}
    This follows by combining \Cref{prop:type III odd cohomolgoy} and \Cref{thm:nagai for type I}.
\end{proof}

We can prove the $\ell$-adic analogue of \Cref{cor:nilpotency of H^3 p-adic} for generalized Kummer types.

\begin{prop}\label{prop:nilpotency of H^3 of generalized Kummer}
Assume that $X$ is of generalized Kummer type
and $\ell \neq p$. 
    If $X$ is of type I, then $\nu(N_{\ell, 3})=0$.
    If $X$ is of type II or type III, then $\nu(N_{\ell, 3})=1$.
\end{prop}

\begin{proof}
    After replacing $K$ by a finite extension, we may assume that we are in the situation of \Cref{Proposition:LLV and Kuga-Satake}.
    So there exists a nilpotent operator
    $N \in \overline{\ge}_\ell(X)$
    which acts on both $H^3_\ell(X)$ and $W\coloneqq H^1_\ell(\KS(X, \scrL)^{\times 2})$ as monodromy operators.
    It follows from \cite[Lemma 3.11]{Soldatenkov20a}
    that 
    $H^3_\ell(X)_{\overline{\Q}_\ell}$
    is a direct sum of the spin representation
    $\overline{V}_{(\frac{1}{2}, \frac{1}{2}, \frac{1}{2}), \overline{\Q}_\ell}$ as a $\rfrg_{\ell}(X)$-module.
    (In fact, we know that $H^3_\ell(X)_{\overline{\Q}_\ell}=\overline{V}_{(\frac{1}{2}, \frac{1}{2}, \frac{1}{2}), \overline{\Q}_\ell}$ since $b_3(X) =8$.)
    Thus the desired description follows from 
    \Cref{thm:l-adic monodromy}
    and
    the fact that
    $W_{\overline{\Q}_\ell}$
    is a direct sum of $\overline{V}_{(\frac{1}{2}, \frac{1}{2}, \frac{1}{2}), \overline{\Q}_\ell}$
    as an $\so(H^2_\ell(X))_{\overline{\Q}_\ell}$-module.
\end{proof}

\subsection*{Acknowledgements}
The authors would like to thank Salvatore Floccari, Lie Fu, Zhiyuan Li, Fuetaro Yobuko for helpful discussions and invaluable comments on a preliminary version of this paper.
K.~Ito is supported by JSPS KAKENHI Grant Numbers 24K16887 and 24H00015.
T.~Ito and T.~Koshikawa are supported by JSPS KAKENHI Grant Number 23K20786. 
T.~Takamatsu is supported by JSPS KAKENHI Grant Numbers 22KJ1780 and 25K17228.
H.~Zou is supported by the Deutsche Forschungsgemeinschaft (DFG, German Research Foundation) – Project-ID 491392403 – TRR 358.

\printbibliography
\end{document}